\documentclass[12pt,oneside]{amsart}
\usepackage{amssymb, amsmath, amsthm}
\usepackage{pinlabel}
\usepackage{pdfpages}
\usepackage{epstopdf}
\usepackage{graphicx}
\usepackage{color}
\usepackage{enumerate}

\usepackage[colorlinks = true,
            linkcolor = red,
            urlcolor  = red,
            citecolor = red,
            anchorcolor = red]{hyperref}
\usepackage{amsrefs}
\usepackage[pagewise]{lineno}

\theoremstyle{plain}
\newtheorem{theorem}{Theorem}[section]

\newtheorem*{theorem*}{Theorem}
\newtheorem*{maintheorem-intro}{Main Theorem}
\newtheorem*{maintheorem-intro-2}{Theorem~\ref{Bridge number and genus}}
\newtheorem*{theorem-cablingconj}{Theorem~\ref{apps1} (1)}
\newtheorem*{theorem-toroidal}{Specialization of Theorem~\ref{apps1} (2)}
\newtheorem*{theorem-lens}{Theorem~\ref{Bounding distance - special}(1)}
\newtheorem*{theorem-SFS}{Theorem~\ref{Bounding distance - special}(2)}
\newtheorem*{theorem-cosmetic}{Theorem}

\newtheorem*{theorem-bridge}{Specialization of Corollary~\ref{Cor: exceptional bridge}}
\newtheorem*{theorem-Heeggenus}{Corollary~\ref{Cor: exceptional bridge} (2)}

\newtheorem{corollary}[theorem]{Corollary}
\newtheorem{lemma}[theorem]{Lemma}

\theoremstyle{definition}
\newtheorem{remark}[theorem]{Remark}
\newtheorem{definition}[theorem]{Definition}

\theoremstyle{definition}

\newcommand{\R}{{\mathbb R}}

\newcommand{\Z}{\mathbb Z}
\newcommand{\N}{\mathbb N}

\newcommand{\up}{\uparrow}
\newcommand{\dn}{\downarrow}
\newcommand{\nil}{\varnothing}

\newcommand{\wihat}{\widehat}
\newcommand{\defn}[1]{\emph{#1}}
\newcommand{\bdd}{\partial}
\newcommand{\boundary}{\partial}

\newcommand{\mc}[1]{\mathcal{#1}}
\newcommand{\ob}[1]{\overline{#1}}
\newcommand{\inter}[1]{\mathring{#1}}

\newcommand{\goodCurves}{\mathbb{P}_\mc{C}}
\newcommand{\goodArcs}{\mathbb{P}_{\mc{AC}}}

\newcommand{\co}{\mskip0.5mu\colon\thinspace}

\begin{document}

   \title[Exceptional Surgeries]{Exceptional and cosmetic surgeries on knots}
   \author{Ryan Blair}
   \email{ryan.blair@csulb.edu}
   \author{Marion Campisi}
   \email{marion.campisi@sjsu.edu}
   \author{Jesse Johnson}
   \email{jjohnson@math.okstate.edu}
   \author{Scott A. Taylor}
   \email{sataylor@colby.edu}
   \author{Maggy Tomova}
   \email{maggy-tomova@uiowa.edu}


\begin{abstract}
We show that the bridge distance of a knot determines a lower bound on the genera of essential surfaces and Heegaard surfaces in the manifolds that result from non-trivial Dehn surgeries on the knot. In particular, knots with high bridge distance do not admit non-trivial non-hyperbolic surgeries or non-trivial cosmetic surgeries. We further show that if a knot has bridge distance at least 3 then its bridge number is bounded above by a function of Seifert genus, or indeed by the genus of (almost) any essential surface or Heegaard surface in the surgered manifold.
\end{abstract}
\maketitle
\date{\today}

\section{Introduction}\label{First Intro}

We begin with a very brief summary of some of the highlights in the paper for easy reference. In the next section we explain the origin and significance of these results in greater detail.

The main goal of this paper is to show that the genera of essential surfaces and Heegaard surfaces in a manifold obtained by Dehn surgery on a knot $L$ are bounded below by a certain function of the bridge distance of $L$.  The bridge distance of a knot $L$ in a 3-manifold $M$, like the bridge number $b(L)$, is a non-negative integer invariant of knots in 3-manifolds. In fact, there are two slightly different measures of bridge distance, $d_\mc{C}$ and $d_\mc{AC}$; we consider both of them.

The lower bound provided by bridge distance shows that a knot with high bridge distance does not admit exceptional or cosmetic surgeries. In particular, we make significant progress towards solving the Cabling Conjecture by showing that if $L \subset S^3$ has a reducing surgery then $d_\mc{C}(L) \leq 2$ (assuming $b(L) \geq 6$). In \cite{Distance2}, we give a characterization of knots $L\subset S^3$ with $d_\mc{C}(L) = 2$. Combined with the fact that if a knot $L \subset S^3$ has $d_\mc{C}(L) = 1$ then it has an essential tangle decomposition \cite{HS01}, this gives a detailed description of any potential counterexample to the conjecture. On a related note, we show that the Berge Conjecture (concerning knots with lens space surgeries) is true for all knots $L \subset S^3$ with $d_\mc{C}(L) \geq 5$. The following theorem provides a summary of these and related results.

\begin{maintheorem-intro}\label{MainThm-intro}
Let $M$ be a closed, orientable, irreducible manifold and suppose that $L \subset M$ is a knot whose exterior is irreducible and $\boundary$-irreducible. Then:
\begin{enumerate}
\item\label{case:reducible} If $M = S^3$, $b(L) \geq 6$, and $d_\mc{C}(L) \geq 3$, then $L$ does not admit a reducible surgery.\\
\item\label{case:toroidal}  If $M = S^3$, $b(L) \geq 7$, and $d_\mc{C}(L) \geq 3$, then  $L$ does not admit a toroidal surgery.\\
\item\label{case:lens space} If $M = S^3$ and $d_\mc{C}(L) \geq 5$, then $L$ does not admit a lens space surgery. \\
\item\label{case:SFS}  If $M = S^3$ and $d_\mc{C}(L) \geq 7$, then $L$ does not admit a small Seifert fibered space surgery.\\
\item\label{case:non-hyperbolic}  If $M$ is hyperbolic and if $L$ is a knot with $d_\mc{C}(L) \geq 13$, then every non-trivial surgery on $L$ produces a hyperbolic 3-manifold.\\
\item\label{case:cosmetic} If $M$ has Heegaard genus $g \geq 1$, and $d_\mc{C}(L) \geq 4g + 5$, then every surgery on $L$ produces a 3-manifold of Heegaard genus at least $g+1$. \\
\item\label{case:S^3} If $d_\mc{C}(L) \geq 7$, then $L$ does not admit a non-trivial surgery producing $S^3$.
\end{enumerate}
\end{maintheorem-intro}
\begin{remark}
Conclusions \eqref{case:reducible} and \eqref{case:toroidal} can be found in Theorem~\ref{apps1}. Conclusions \eqref{case:lens space} - \eqref{case:cosmetic} are contained in Theorem~\ref{Bounding distance - special} below. Conclusion (\ref{case:S^3}) can be found in Corollary \ref{Cor: S^3 surgeries}.
\end{remark}

In fact, it turns out that there is a surprising relationship between bridge number, bridge distance, and the genera of essential surfaces in the knot exterior. Using this relationship, we also exhibit a startling connection between the Seifert genus and bridge number of a knot of high bridge distance:
\begin{maintheorem-intro-2}
Suppose that $L$ is a knot in a homology sphere with irreducible and $\boundary$-irreducible exterior.  If $d_\mc{C}(L) \geq 3$, then
\[
b(L) \leq 4g(L)+2
\]
where $g(L)$ is the Seifert genus of $L$.
\end{maintheorem-intro-2}

\section{Background and results}
Every link $L$ in a 3-manifold $M$ can be put into bridge position with respect to a minimum genus Heegaard surface $H$ for $M$ \cite{Doll}. The bridge number of $L$ is a well-studied, natural number invariant of the link (see Section \ref{bridge surfaces defn} below). As with Heegaard splittings \cite{Hempel}, almost every bridge surface can be uniquely assigned a non-negative integer called ``distance'' (see Section~\ref{sec: dist}). In an effort to provide the strongest possible results, we study two related notions of distance. Both of our notions of distance are minor variations of those introduced by Saito \cite{Saito} and Bachman-Schleimer \cite{BS}. The distance in the curve complex $d_\mc{C}(L)$ will be defined to be the supremum of the distances among minimal bridge surfaces as measured in the curve complex. This distance $d_\mc{C}(L)$ is a link invariant lying in $\Z \cup \{-\infty, \infty\}$. It turns out, as we explain below, that (after adding the assumptions that $M$ is irreducible, that $L$ is a knot, and that if $M = S^3$, then $b(L) \geq 3$) we actually have $d_\mc{C}(L) \in \N$.  Similarly, the distance in the arc and curve complex $d_\mc{AC}(L)$ will be defined to be the supremum of the distances among minimal bridge surfaces as measured in the arc and curve complex. Like distance in the curve complex, $d_\mc{AC}$ takes values in $\N$ for most knots in irreducible 3-manifolds. As we will be producing upper bounds on quantities involving distance, we note that there exist knots and links of arbitrarily high distance \cites{BTY, IS, JM}. Indeed, if the well-known analogy between bridge surfaces and Heegaard surfaces continues to hold, high distance knots are likely ``generic'' \cites{Maher, LM}. Finally, we note that work of Hayashi-Shimokawa \cite{HS} and Taylor-Tomova \cite{TT} give a useful decomposition theorem for links $L$ with $d_\mc{C}(L) = 1$; Saito gives a classification of (1,1) knots $L \subset S^3$ with $d_\mc{C}(L) \leq 2$; and we \cite{Distance2} give a detailed description of knots $L\subset S^3$ with $d_\mc{C}(L) = 2$.

Although bridge number and distance are independent of each other, we find a surprising relationship between their product and the genus of surfaces in 3-manifolds obtained by Dehn surgery on $L$. More precisely, we show:

\begin{theorem}\label{Main-intro}
Suppose that $L$ is a knot with irreducible and $\boundary$-irreducible exterior in a closed, orientable 3-manifold $M$.  Let $M'$ be the result of non-trivial Dehn surgery on $L$. Let $S \subset M'$ be a closed connected orientable surface of genus $g$. Then the following hold:
\begin{enumerate}
\item If $S$ is an essential surface and if $\Delta$ is the surgery distance, then either there is a closed essential surface of genus $g$ in the exterior of $L$ or
\[
\Delta b(L)(d_\mc{AC}(L) - 2) \leq \max(1,4g - 2)
\]
\item If $S$ is a Heegaard surface, then one of the following holds:
\begin{itemize}
\item $b(L)(d_\mc{AC}(L) - 2) \leq \max(1, 2g)$
\item There is a Heegaard surface of genus $g$ for the exterior of $L$
\item $M$ contains an essential surface of genus strictly less than $g$ which intersects $L$ at most twice.
\item $M(L)$ contains an essential surface of genus at most $g-1$ and $L^*$, the surgery dual to $L$, is isotopic into the surface. Furthermore, the dual surgery on $L^*$ creating $M$ from $M'$ has surgery slope equal to the slope of the boundary of the surface after $L^*$ has been isotoped to lie in it.
\end{itemize}
\end{enumerate}
\end{theorem}

The first conclusion is proved as part of Theorem~\ref{Thm: Ess Surface} and the second as part of Theorem~\ref{Thm:Nonmerid bridge}. Refinements of the bounds in Theorem~\ref{Main-intro} give very strong bounds on the distance and bridge numbers of hyperbolic knots $L$ with exceptional or cosmetic surgeries. We explore these consequences, and others, in the subsections which follow.

\subsection{Cabling Conjecture}\label{sec:Cabling}
The Cabling Conjecture \cite{GAS} asserts that if a knot in $S^3$ has a reducing surgery, then the knot is a cable knot and the slope of the surgery is the slope of the cabling annulus. The conjecture has been shown to hold for many classes of knots including satellite knots~\cite{S}, symmetric knots~\cites{EM92, HS,LZ}, persistently laminar knots~\cites{B1,B2}, alternating knots~\cite{MT}, many knots with essential tangle decompositions~\cite{H}, knots where a surgery produces a connected sum of lens spaces~\cite{Gr}, knots with bridge number at most 4 \cites{Hoffman}, and knots that are band sums~\cite{T}. In this paper, we show knots $L$ with $b(L) \geq 6$ and $d_\mc{C}(L) \geq 3$ satisfy the Cabling Conjecture (see Main Theorem, part \ref{case:reducible}).

This puts strong restrictions on any potential counterexample to the Cabling Conjecture. For, suppose that a counterexample $L \subset S^3$ exists. Hoffman \cite{H1} showed that $b(L) \geq 5$ and, in \cite{Hoffman}, claims he has also proved (in unpublished notes) that $b(L) \geq 6$. Grove \cite{Grove} has recently extended Hoffman's work and has proven that $b(L) \geq 6$. Our result, together with Hoffman and Grove's results, reduces the Cabling Conjecture to studying knots with $d_\mc{C}(L) \leq 2$. Cable knots have distance at most 2, so we have substantial new evidence for the Cabling Conjecture. In fact, it suggests a program for proving the Cabling Conjecture: First, prove it for knots $L$ with $d_\mc{C}(L) = 1$ (which is partially done by Hayashi \cite{H}, since such knots have an essential tangle decomposition). Second, prove it for knots $L$ with $d_\mc{C}(L) = 2$. Such knots have been extensively studied in \cite{Distance2}.

\subsection{Toroidal surgeries}
A characterization of hyperbolic knots in $S^3$ having toroidal surgeries is more elusive and examples are easily constructed by considering knots lying on knotted genus two surfaces in $S^3$. Additionally, many other examples of knots with toroidal surgeries are known (e.g.~\cites{EM93, T03}). It is known that the surgery slope of a toroidal surgery must be integral or half-integral ~\cites{GL95} and the punctured torus in the knot exterior can be assumed to have no more than 2 boundary components~\cite{GL00}. Furthermore, in the case when the surgery slope is non-integral, the knot must be a ``Eudave-Mu\~noz knot'' \cite{GL04}. We show that no hyperbolic knot $L\subset S^3$ with $b(L) \geq 7$ and $d_\mc{C}(L) \geq 3$ has a surgery producing a toroidal 3-manifold (see Main Theorem, part \ref{case:toroidal}).

\subsection{Berge Conjecture}
The Berge Conjecture (cf.~\cite[Problem 1.78]{K}) states that every knot in $S^3$ with a lens space surgery is doubly primitive with respect to a genus two Heegaard surface for $S^3$.  Recently, a number of research programs to prove the Berge Conjecture have been proposed and partially completed. In~\cite{BGH}, a two step program using knot Floer homology is outlined and the first step of the program is completed in~\cite{He}. In~\cite{W}, the first step in a three step program is completed that would result in the proof of the Berge Conjecture for tunnel number one knots. We show that no surgery on a knot $L \subset S^3$ with $d_\mc{C}(L) \geq 5$ produces a lens space (see Main Theorem, part \ref{case:lens space}.)

\subsection{Small Seifert fibered spaces and other exceptional surgeries}
Using the fact that small Seifert fibered spaces have Heegaard genus at most 2~\cite{BZ}, we show that no non-trivial surgery on a knot $L \subset S^3$ with $d_\mc{C}(L) \geq 7$ produces a small Seifert-fibered space (see Main Theorem, part \ref{case:SFS}.)

Theorem~\ref{Bounding distance - special} also gives a version of these results for hyperbolic knots in other 3-manifolds. It uses the Geometrization Theorem~\cites{P1,P2,P3} to show that a knot $L$ in a hyperbolic manifold with $d_\mc{C}(L) \geq 13$ does not have any non-hyperbolic surgeries.

Knots of distance at most 12 are certainly plentiful. But, as long as the distance of a hyperbolic knot $L$ admitting an exceptional surgery is at least 3, we can also often bound $b(L)$. We show (Corollary \ref{Cor: exceptional bridge}) that if $L \subset S^3$ has $b(L) \geq 9$ and $d_\mc{C}(L) \geq 3$ and has tunnel number at least 2, then $L$ admits no non-hyperbolic surgery.

\subsection{Heegaard genus}
Hyperbolic knots not only have very few exceptional surgeries, they also have very few cosmetic fillings~\cite[Theorem 1]{BHW}, i.e., distinct surgeries producing homeomorphic manifolds. A cosmetic surgery on a knot is a non-trivial surgery returning the original manifold. It is conjectured~\cite[Problem 1.81]{K} that hyperbolic knots have no so-called ``exotic'' cosmetic surgeries. For example, Gabai~\cite{G} showed that no non-trivial knot in $S^2 \times S^1$  admits a non-trivial cosmetic surgery. Gordon and Luecke's solution to the knot complement conjecture shows that non-trivial knots in $S^3$ also admit no cosmetic surgeries~\cite{GL89}. The article~\cite{BHW} gives good introduction to cosmetic surgery.

We show:

\begin{theorem-cosmetic}[Theorem~\ref{Bounding distance - special}(4) and Corollary~\ref{Cor: exceptional bridge}]
Assume that $M$ is closed, orientable, irreducible and has Heegaard genus $g$. Let $L \subset M$ be a knot with irreducible and $\boundary$-irreducible exterior. If
\[ d_{\mc{C}}(L) \geq \max(7,4g+ 5)\]
then $L$ does not admit non-trivial cosmetic surgeries. Furthermore, if $M$ is non-Haken, if $M(L)$ does not contain an essential surface of genus at most $g-1$, if $d_\mc{C}(L) \geq 3$, and if $b(L) \geq 2g + 5$, then $L$ also does not admit a non-trivial cosmetic surgery.
\end{theorem-cosmetic}

In particular, as stated in the Main Theorem, part \eqref{case:S^3}, a knot $L$ with irreducible and $\boundary$-irreducible exterior and with $d_\mc{C}(L) \geq 7$ admits no $S^3$ surgery. This stops short of giving a new proof of the Gordon-Luecke result, but does constrain knots in other manifolds having $S^3$ surgeries.  Our theorem is proved by using the Heegaard genus of a manifold resulting from Dehn surgery on $L$ to bound the bridge distance and bridge number of $L$. Many authors (eg. \cite{MR, R, RiSe, FP}) have studied the effect of Dehn surgery on Heegaard surfaces. Most of them have fixed a knot and studied the Heegaard genera which can result from surgery on the knot, often producing an inequality relating surgery distance to Heegaard genus. That is, they show that given a knot, most \emph{slopes} do not admit Heegaard genus decreasing surgeries. On the other hand, we show that most \emph{knots} do not admit Heegaard genus decreasing surgeries. A philosophy similar to the one presented in the current paper is evident in the recent paper by Baker, Gordon, and Luecke ~\cite{BGL} that relates surgery distance, knot width, and the Heegaard genus of the surgered 3-manifold for certain knots in $S^3$.

\subsection{Genus and bridge number}
As a final application, we consider the relationship of Seifert genus to bridge number.

In general, the Seifert genus and the bridge number of a knot $L \subset S^3$ are unrelated. For example, let $J$ be a non-trivial knot and let $K_n$ be the $n$-th Whitehead double of $J$. The knots $K_n$ each have Seifert genus 1, but have bridge number going to infinity with $n$ \cites{Schubert, Schultens}. Hyperbolic examples of this phenomena can likely be constructed using recently developed machinery of Baker, Gordon, and Luecke \cite{BGL2}. In stark contrast, this cannot occur when the knots have bridge distance at least 3:

\begin{maintheorem-intro-2}
Suppose that $L$ is a knot in a homology sphere $M$ with irreducible and $\boundary$-irreducible exterior. Then either $d_\mc{C}(L) \leq 2$ or
\[
b(L) \leq 4g(L) +2,
\]
where $g(L)$ is the Seifert genus of $L$.\end{maintheorem-intro-2}

\subsection{Structure of the paper}

In Section~\ref{sec:definitions} we introduce relevant definitions and background results.  In Section~\ref{sec:essential surfaces}, we bound quantities involving bridge number and bridge distance of a bridge surface for a knot in terms of the genus of surfaces $S$ in the knot exterior that admit desirable embeddings. In Section~\ref{Non-MeridoinalEss}, we apply these bounds in the case when $S$ is an essential surface. In Section~\ref{sec:doublesweepout}, we analyze the graphic associated to the simultaneous sweepouts by two  differently sloped Heegaard surfaces. In Section~\ref{sec: bridge surface bounds} and Section~\ref{sec:ImprovingBridgeBound}, we apply the bounds when $S$ is a Heegaard surface for the surgered manifold (what we call an alternately sloped Heegaard surface).

\section{Definitions}
\label{sec:definitions}

If $X$ is a properly embedded submanifold of $M$, we let ${\eta}(X)$ denote a closed regular neighborhood of $X$ and we let $\inter{\eta}(X)$ denote the interior of the closed regular neighborhood (which is, therefore, an open regular neighborhood.) The \defn{boundary of ${\eta}(X)$} is equal to $ \boundary \eta(X) = {\eta}(X)\setminus \inter{\eta}(X)$ and the \defn{frontier of ${\eta}(X)$} is equal to $\boundary \eta(X) \setminus \boundary M$. Thus the frontier of a regular neighborhood of properly embedded disk in a 3-manifold $M$ consists of 2 disks while the boundary is a 2-sphere. The notation $|X|$ indicates the number of components of $X$. If $L$ is a 1-manifold in $M$, then $M(L)$ denotes the exterior $M \setminus \inter{\eta}(L)$ of $L$. We let $\boundary_0 M(L) = \boundary M$ and $\boundary_L M(L) = \boundary M(L) \setminus \boundary_0 M(L)$. Usually we will make the convention that if $\ob{T} \subset M$ is a surface transverse to $L$, then $T$ (dropping the bar) is the surface $\ob{T} \cap M(L)$. We consider the points $\ob{T} \cap L$ as marked points on $\ob{T}$, and when we switch to considering $T$, the marked points correspond to boundary components of $T$.

Suppose that $\ob{F}$ is a properly embedded compact orientable surface in a 3-manifold $M$ which is transverse to a properly embedded 1-manifold $L$. A simple closed curve in $\ob{F}$ is an \defn{essential curve} in $\ob{F}$ if it is disjoint from the marked points, does not bound a disc in $\ob{F}$ having exactly 0 or 1 marked points, and is not parallel to a component of $\boundary \ob{F}$ in the complement of the marked points. Similarly, suppose that $\alpha \subset \ob{F}$ is one of the following:
\begin{itemize}
\item an embedded arc $\alpha \subset \ob{F}$ with endpoints on the union of $\boundary \ob{F}$ and the marked points, having interior disjoint from the marked points, and for which there is no subarc $\beta \subset \boundary \ob{F}$ such that $\boundary \beta = \boundary \alpha$ and $\alpha \cup \beta$ bounds a disc in $\ob{F}$ containing no marked points.
\item an embedded simple closed curve in $\ob{F}$ intersecting exactly one marked point and which does not bound a disc in $\ob{F}$ having no marked points in its interior.
\end{itemize}
Then we call $\alpha$ an \defn{essential arc} in $\ob{F}$. There is a natural bijection between essential curves and arcs in $\ob{F}$ and those in $F$.

For a surface $\ob{F} \subset M$ transverse to $L$, a \defn{compressing disc} $D$ for $\ob{F}$ is a disc embedded in $M$, with interior disjoint from $\ob{F} \cup L$, and having boundary an essential curve on $\ob{F}$. If there is a compressing disc for $\ob{F}$, then $\ob{F}$ is \defn{compressible}. Otherwise, $\ob{F}$ is \defn{incompressible}. A disc $D \subset M$ with embedded interior disjoint from $L \cup \ob{F}$ and whose boundary is the endpoint union of an arc in $\boundary M \cup L$ and an essential arc $\alpha$ in $\ob{F}$ is called a \defn{$\boundary$-compressing disc} for $\ob{F}$. Abusing terminology slightly, we say that $\alpha$ \defn{bounds} the disc $D$. If there is a $\boundary$-compressing disc for $\ob{F}$ then we say that $\ob{F}$ is \defn{$\boundary$-compressible} and if there is no such disc then $\ob{F}$ is \defn{$\boundary$-incompressible}. Observe that $\ob{F}$ is compressible if and only if $F$ is compressible in $M(L)$ and that $\ob{F}$ is $\boundary$-compressible  if and only if $F$ is $\boundary$-compressible in $M(L)$.

A surface $F$ properly embedded in a 3-manifold $N$ is an \defn{essential surface} if it is incompressible, not boundary parallel, and not a 2--sphere bounding a 3-ball. Unless explicitly stated otherwise, all essential surfaces will be surfaces without marked points (but could have non-empty boundary).

If $V$ is a union of tori then a \defn{multislope} $\sigma$ in $V$ is an isotopy class in $V$ of a union of essential curves, one in each component of $V$. If $V \subset \boundary N$ and if $F \subset N$ is a properly embedded surface, we say that $F$ defines a $\sigma$-slope in $V$ (and that $F$ is $\sigma$-sloped) if $\sigma$ is represented by the union of components of $\boundary F \cap V$. If $\sigma$ and $\tau$ are multislopes in $V$, their \textit{intersection number}, denoted $\Delta(\sigma,\tau)$, is calculated by taking the minimum intersection number (over all components of $V$) between minimally intersecting representatives of $\sigma$ and $\tau$ in a single component of $V$. In particular, if $\Delta(\sigma, \tau) > 0$ then the slopes are distinct in every component. The result of Dehn filling $N = M(L)$ along the multislope $\sigma$ is denoted $M(L)(\sigma)$. When $L$ and $\sigma$ are clear from context, we will use $M'$ instead of $M(L)(\sigma)$ to avoid cumbersome notation. The \defn{surgery distance} between $M(L)(\sigma)$ and $M(L)(\tau)$ is the value of $\Delta(\sigma,\tau)$. When $\sigma$ and $\tau$ are understood, we will simply write $\Delta$.

\subsection{Heegaard splittings}
A \defn{handlebody} is a compact 3-manifold homeomorphic to a closed regular neighborhood of a finite graph $G$ embedded in $\R^3$. A \defn{compressionbody} $C$ is a connected 3-manifold homeomorphic to any component of a regular neighborhood in a orientable 3-manifold $N$ of $G \cup \partial N$, where $G$ is a properly embedded finite graph in $N$ (i.e., $\partial N \cap G$ consists only of vertices of $G$). We allow $G$ to be empty and thus a compressionbody may be a regular neighborhood of a surface, i.e., a product. If $G$ is disjoint from $\boundary N$, then every component of $\eta(G)$ is a handlebody, so handlebodies are examples of compressionbodies. The intersection $C \cap \partial N$ is called the \defn{negative boundary} of $C$ and is denoted $\boundary_- C$. We call $\boundary_+ C = \boundary C \setminus \boundary_- C$ the \defn{positive boundary} of $C$. There is a collection $\mc{D}$ of properly embedded, pairwise disjoint, discs in $C$ with boundary on $\boundary_+ C$ such that the result of $\boundary$-reducing  $C$ using $\mc{D}$ is isotopic to $\boundary_-C \times I$ (where $I$ is the interval $[0,1]$). This product $\boundary_- C \times I$ is not uniquely defined as it depends on the choice of discs $\mc{D}$, but it is an embedded (though not properly embedded) submanifold of $C$.

A \defn{Heegaard splitting} for a 3-manifold $M$ is a triple $(\ob{H}, \ob{H}_{\dn}, \ob{H}_{\up})$ where $\ob{H}$ is a connected, closed, embedded, separating surface and $\ob{H}_{\dn}$, $\ob{H}_{\up}$ are compressionbodies with disjoint interiors such that $M=\ob{H}_{\dn}\cup_{\ob{H}} \ob{H}_{\up}$ and $\ob{H} = \boundary_+ \ob{H}_\dn = \boundary_+ \ob{H}_\up$. The surface $\ob{H}$ is called a \textit{Heegaard surface}.

\subsection{Bridge surfaces}\label{bridge surfaces defn}

The terminology presented in this section is an adaptation of that used by various authors including Hayashi-\linebreak Shimokawa~\cite{HS01}, Scharlemann-Tomova~\cite{STo}, and Taylor-Tomova~\cite{TT}.

Let $\ob{H}$ be a Heegaard surface for a compact, orientable manifold $M$ transverse to a properly embedded 1-manifold $L$. $\ob{H}$ is a \defn{bridge surface} for $(M,L)$ if, for each compressionbody $\ob{C}$ which is the closure of a component of $M \setminus \ob{H}$, the arcs $L \cap \ob{C}$ can be properly isotoped (relative to their endpoints) so that each either lies in $\ob{H}$ or is vertical in $\boundary_- \ob{C} \times I$ where, as above, $\boundary_- \ob{C} \times I$ is the result of $\boundary$-reducing $\ob{C}$ using some collection $\mc{D}$ of properly embedded, pairwise disjoint, discs in $\ob{C}$ which are disjoint from $L$ and have boundary on $\boundary_+ \ob{C}$.

If $\ob{H}$ is a bridge surface for $(M, L)$ and if $\alpha$ is the closure of a component of $L \setminus \ob{H}$, then either $\alpha$ has both endpoints on $\ob{H}$ or $\alpha$ has one endpoint on $\ob{H}$ and one on $\boundary M$. In the former case, $\alpha$ is a \defn{bridge arc} and in the latter case $\alpha$ is a \defn{vertical arc}. For each bridge arc $\alpha$ in $\ob{C}$, there is an arc $\beta$ in $\ob{H}$ having the same endpoints as $\alpha$ and with interior disjoint from $L$ so that $\alpha \cup \beta$ bounds a disc in $\ob{C}$ with interior disjoint from $L$. Such a disc is called a \defn{bridge disc} for $\alpha$. Note that a bridge disk is a boundary compressing disc for $\ob{H}$. If $\ob{H}$ is a bridge surface for $(M,L)$, we say that $L$ is in \defn{bridge position} with respect to $\ob{H}$.

If $\ob{H}$ is any bridge surface for $(M,L)$, we define the \defn{bridge number} $b(H) = b(\ob{H}) = |\ob{H} \cap L|/2$. The \defn{bridge number} $b(L)$ of a knot $L \subset M$ is the minimum value of $b(\ob{H})$ taken over all minimal genus bridge surfaces of $(M,L)$. On a related note, we will also need to use the following notion: Assume $\ob{H}$ is a bridge surface for $(M,L)$. For each component $K$ of $L$, count the number of intersections of $\ob{H}$ with $K$. Let $b_{\min}(H)$ be half the minimum, taken over all components of $L$. Since $\ob{H}$ is a bridge surface for $(M,L)$, the number $b_{\min}(H)$ is a natural number and if $L$ is a knot then $b_{\min}(H) = b(H)$.

A bridge surface $\ob{H}$ is \defn{stabilized} if there is a pair of compressing discs $D_{\dn}$, $D_{\up}$ on opposite sides of $\ob{H}$ such that $\boundary D_{\dn}$ and  $\boundary D_{\up}$ intersect transversally and in a single point. A bridge surface $\ob{H}$ is \defn{weakly reducible} if there is a pair of disjoint compressing discs $D_{\dn}$, $D_{\up}$ on opposite sides of $\ob{H}$. A bridge surface, which is compressible to both sides but is not weakly reducible is called \defn{strongly irreducible}.  A bridge surface $\ob{H}$ is \defn{weakly $\boundary$-reducible} if there is a pair of disjoint bridge discs or compressing discs on opposite sides of $\ob{H}$. If $\ob{H}$ has bridge discs or compressing discs on opposite sides but is not weakly $\boundary$-reducible we say it is \defn{strongly $\boundary$-irreducible}.

Two ways of being weakly $\boundary$-reducible are particularly important to us in the proof of Theorem \ref{Thm:Nonmerid bridge}. Suppose that a bridge surface $\ob{H}$ has bridge discs $D_\dn$ and $D_\up$ on opposite sides of $\ob{H}$ such that $D_\dn \cap D_\up$ is contained in $L$. If $D_\dn \cap D_\up$ is a single point, then $\ob{H}$ is \defn{perturbed}. If $D_\dn \cap D_\up$ is two points, then $\ob{H}$ is \defn{cancellable} and the component $K = L \cap (\boundary D_\dn \cup \boundary D_\up)$ is called a \defn{cancellable component} of $L$. It is possible that there are multiple cancellable components. If $\ob{H}$ is perturbed, isotoping $\ob{H}$ across either $D_\up$ or $D_\dn$ produces a new bridge surface for $(M,L)$ intersecting $L$ two fewer times. If $\ob{H}$ is cancellable and if $K$ is a cancellable component, we may isotope $\ob{H}$ in a regular neighborhood of $D_\dn \cup D_\up$ to contain $K$. We call the surface $\ob{R}$ resulting from this isotopy a \defn{cancelled bridge surface} and set $R = \ob{R} \cap M(L)$. We will say more about these surfaces later. It is easy to see that a bridge surface which is weakly $\boundary$-reducible is either weakly reducible, perturbed, or cancellable.

We say that a component $K$ of $L$ is \textit{removable} if $K$ is cancellable with cancelling discs $D_\dn, D_\up$ and there is a compressing disc $D$ disjoint from one of $D_\dn, D_\up$ and intersecting the other transversally in a single point in its boundary. In this case we also say that the bridge surface $\ob{H}$ is \defn{removable}. If a component $K \subset L$ is removable, the surface $\ob{H}$ can be isotoped in a regular neighborhood of $D \cup D_\dn \cup D_\up$ to be a bridge surface for $(M \setminus \inter{\eta}(K), L \setminus K)$. See~\cite{STo} for more details.

As usual, if $\ob{H}$ is stabilized, weakly reducible, etc. then we say that $H = \ob{H} \cap M(L)$ is as well.

\subsection{Thin position}
Ever since Gabai's introduction \cite{G} of thin position, it has been an extremely useful tool in studying knots and 3-manifolds. We use a version of thin position due, in its original form, to Hayashi and Shimokawa \cite{HS}. Here are the relevant definitions and results.

Let $(\ob{\mc{F} }, \ob{\mc{S}}) \subset M$ be a pair of (not necessarily connected) orientable, closed, properly embedded surfaces transverse to $L$. Then $(\ob{\mc{F}}, \ob{\mc{S}})$ is a \defn{multiple bridge surface} \cite{HS} for $(M,L)$ if $\ob{\mc{F}}$ is separating, in each component $M_i$ of $M \setminus \inter{\eta} (\ob{\mc{F}})$ there is a unique component of $\ob{\mc{S}}$, and that component is a bridge surface for $(M_i, M_i \cap L)$ and for each component of $\ob{\mc{S}}$ there is such an $i$. We call the collection $\ob{\mc{F}}$ the \defn{thin surfaces} and the collection $\ob{\mc{S}}$ the \defn{thick surfaces}. Since each thick surface is a bridge surface for a component of $M \setminus \inter{\eta}(F)$, we will apply the terms ``stabilized'', ``weakly reducible'', etc. to the thick surfaces as necessary.

Taylor and Tomova \cite{TT} (generalizing work of Hayashi-Shimokawa)  prove a more general version of the following theorem.

\begin{theorem}\label{ThinPosition1}
Let $L \subset M$ be a link in a closed, orientable 3-manifold such that $M(L)$ is irreducible and no sphere in $M$ intersects $L$ transversally exactly once. Suppose that $\ob{H}$ is a genus $g$ bridge surface for $(M,L)$ which is not removable, perturbed, or stabilized. Then there is a multiple bridge surface $(\ob{\mc{F}}, \ob{\mc{S}})$ for $(M,L)$ such that all of the following hold:
\begin{enumerate}
\item Each component of $\mc{F} = \ob{\mc{F}} \cap M(L)$ is essential in $M(L)$.
\item Each component of $\ob{\mc{S}}$ is a strongly irreducible bridge surface for $(M_0, L \cap M_0)$ where $M_0$ is the component of the closure of $M \setminus \inter{\eta} (\ob{\mc{F}})$ containing it.
\item Each component $\ob{J}$ of $\ob{\mc{S}} \cup\ob{\mc{F}}$ has genus no greater than $g$ and intersects $L$ no more than $|\ob{H} \cap L|$ times and is obtained by compressing $\ob{H}$ in $M(L)$. Furthermore, if $\ob{\mc{F}}$ is non-empty, at least one compression of $\ob{H}$ is necessary to form any component of $\ob{\mc{F}} \cup \ob{\mc{S}}$.
\item No component of $\ob{\mc{S}}$ is perturbed or removable in its component of the closure of $M \setminus \inter{\eta} (\ob{\mc{F}})$.
\item Each component of $\ob{\mc{S}}$ is separating in $M$.
\end{enumerate}
\end{theorem}
\begin{proof}
Conclusions (1) and (2) follow directly from \cite[Corollary 9.4]{TT} applied with $K = \ob{H}$, $T = L$, and $\Gamma = \nil$. Conclusions (3) and (5) follow from the definition of ``thinning'' given in \cite{TT}. Conclusion (4) follows from \cite[Lemma 6.3]{TT}.
\end{proof}

\begin{remark}
As stated above, Theorem~\ref{ThinPosition1} might also follow from the work of \cite{HS}, although it is not immediately clear that all of Hayashi and Shimokawa's thin surfaces are essential. It is also similar to Campisi's theorem in \cite{C} which develops thin position for sloped Heegaard splittings (see below). Her theorem, however, also does not guarantee that all thin surfaces are essential.
\end{remark}

The next lemma records, for later reference, some properties of a cancelled bridge surface.
\begin{lemma}\label{Lem:Cancellable}
Let $L \subset M$ be a link in a closed, orientable 3-manifold such that $M(L)$ is irreducible and no sphere in $M$ intersects $L$ transversally exactly once.  Let $(\ob{\mc{F}}, \ob{\mc{S}})$ be a multiple bridge surface for $(M,L)$ satisfying conclusions (1), (2), (4), and (5) of Theorem \ref{ThinPosition1}. Suppose that a component $ \ob{J} \subset \ob{\mc{S}}$ is a cancellable bridge surface for $(M_0, L \cap M_0)$ where $M_0$ is the closure of the component of $M\setminus \inter{\eta}(\ob{\mc{F}})$ that contains $\ob{J}$. Let $K$ be a cancellable component of $L\cap M_0$ and let $\ob{R}$ be a cancelled bridge surface obtained by isotoping $\ob{J}$ to contain $K$. Then all of the following hold:
\begin{itemize}
\item Let $\rho$ and $\tau$ be slopes in $\partial \eta(K)$ defined by $\boundary R$ and $\boundary J$, respectively. Then $\Delta(\rho, \tau)=1$.
\item $R$ is an essential surface in $M(L)$.
\item The sum of the genera of the components of $R$ is equal to $g(J)$ if $R$ disconnected and equal to $g(J)-1$ if $R$ is connected.
\item If $K = L$, then one of the following happens:
\begin{enumerate}
\item\label{incomp surf} after Dehn filling $\boundary \eta(K) \subset \boundary M(L)$ with slope $\rho$ and capping $\boundary R$ with discs, we obtain an incompressible surface in the filled manifold.
\item\label{isotopic thin} $K$ is isotopic into a component $\ob{F}$ of $\ob{\mc{F}}$. Furthermore, in $M$, $\ob{F}$ and $\ob{J}$ bound a product compressionbody.
\end{enumerate}
\end{itemize}
\end{lemma}

\begin{proof}
Suppose that $D_{\dn}$ and $D_{\up}$ are bridge discs for the cancellable component $K \subset \boundary D_\dn \cup \boundary D_\up$. We can obtain $R$ by simultaneously boundary compressing $J$ using the discs $D_\dn$ and $D_\up$. Since these discs each intersect a component of $\partial J$ exactly once, each component of $\boundary R$ in $\boundary \eta(K)$ intersects a meridian of $K$ exactly once. This is the first claim.

Suppose that $E$ is a compressing disc for $R$. Since $E$ is disjoint from $K$, we can assume $E$ is fixed by the isotopy taking $\ob{R}$ to $\ob{J}$ and that $E$ is disjoint from $D_\dn\cup D_\up$.

Suppose $\partial E$ is inessential in $J$, then $\partial E$ bounds a disc in $\ob{R}$ that is disjoint from $L$. Since $E$ is disjoint from $D_\dn\cup D_\up$ and $K\subset D_\dn\cup D_\up$, then $D_\dn\cup D_\up$ is disjoint from the disc $\partial E$ bounds in $J$. Thus, $\partial E$ is inessential in $R$, a contradiction to the fact that $M(L)$ is $\boundary$-irreducible. Hence, $E$ is a compressing disc for $J$. Since by assumption $\mc{F}$ is incompressible (or empty), we may isotope $E$ to lie in $M_0$.  Since one of $D_{\dn}$, $D_{\up}$ is on the opposite side of $\ob{J}$ from $E$ and since $E$ is disjoint from $D_\dn\cup D_\up$, then $J$ is weakly reducible, a contradiction to our assumption that the thick surfaces satisfy the conclusions of Theorem~\ref{ThinPosition1}. Therefore $R$ must be incompressible.

Suppose that $E$ is a boundary compressing disc for the incompressible surface $R$. Since $\boundary R$ lies in torus components of $\boundary (M(L))$, this implies that $R$ is a boundary parallel annulus. Hence, $L = K$ is a knot. Moreover, $J$ is obtained from $R$ by attaching the annulus $\eta(K) \cap J$ and so $J$ is a torus and $K$ can be isotoped to lie on this torus. In fact, $E$ is a compressing disc for $L$ intersecting one of $D_\dn$ or $D_\up$ exactly once. Hence, $\ob{J}$ is removable, a contradiction to our  Conclusion (4) of  Theorem~\ref{ThinPosition1}. We conclude that $R$ is an essential surface in $M(L)$. This is the second claim.

The loop $K$ in $\ob{J}$ is either separating or non-separating. If it is non-separating, then $R$ is connected and has genus one less than the genus of $J$. If $K$ is separating, then $R$ has two components and their genera add up to the genus of $J$ by the additivity of Euler characteristic. This is the third claim.

Suppose that $K = L$. Since $M(L)$ is irreducible and $\bdd$-irreducible, $K$ is an essential loop in $\ob{R}$. By Conclusion (5) of Theorem~\ref{ThinPosition1}, $\ob{R}$ separates $M$. Suppose that $\ob{R}$ is compressible to both sides in $M$. Recall from the second claim that $R$ is incompressible in $M(L)$. Hence, the Jaco handle addition theorem~\cite{J}, applied to the 3-manifolds on either side of $\ob{R}$, implies that $R(\rho)$ is incompressible in $M(L)(\rho)$. This is Conclusion \eqref{incomp surf}. Assume, therefore, that $\ob{R}$ is incompressible to one side. Since $M$ is closed and since $\ob{R}$ is isotopic in $M_0$ to $\ob{J} \subset \ob{S}$, there must be a component $\ob{F} \subset \ob{\mc{F}}$ which bounds a product compressionbody with $\ob{R}$.  Isotoping $K$ through the product compressionbody puts it in $\ob{F}$. This is Conclusion \eqref{isotopic thin}.
\end{proof}

\subsection{Sloped Heegaard surfaces}\label{sec:sloped}
Rather than working with a (3-manifold, link) pair $(M,L)$, it is often advantageous to work entirely in the exterior of the link. To that end, we follow Campisi~\cite{C} and define the notion of a ``sloped Heegaard surface''. The definition of a ``sloped Heegaard surface''  used in ~\cite{C} is different from the definition we give here, however, one can verify that they are equivalent (but we will not use that fact).

\begin{definition}
Let $N=M(L)$ and let $\boundary_L N=\bdd N \setminus \bdd M$ and let $\sigma$ be a multislope in $\boundary_L N$. Let $L^*$ be the cores of the attaching solid tori for $M(L)(\sigma)$. Meridians of $L^*$ correspond to the slopes of $\sigma$. A \emph{sloped Heegaard surface} for $N$ of slope $\sigma$ is the restriction of any bridge surface for $(M(L)(\sigma), L^*)$ to $N$. We will denote this sloped Heegaard surface by $S$ where $S$ is the restriction of the bridge surface $\ob{S}$ to $N$. We let $S_{\dn}$ and $S_{\up}$ be the restrictions of the compressionbodies $\ob{S}_\dn$ and $\ob{S}_\up$ to $N$.

Similarly, a \emph{sloped generalized Heegaard surface} for $N$ of slope $\sigma$ is the restriction of any multiple bridge surface $(\mc{\ob{F}} , \mc{\ob{S}})$ for $(M(L)(\sigma),L^*)$ to $N$. We will denote this sloped Heegaard surface by $(\mc{F} , \mc{S})$ where $\mc{F} = \ob{\mc{F}} \cap N$ and $\mc{S} = \ob{\mc{S}} \cap N$. We refer to the components of $\mc{F}$ as thin surfaces and the components of $\mc{S}$ as thick surfaces. We define each component of $\mc{S}$ to be a sloped Heegaard surface for the component of $N \setminus \inter{\eta}(\mc{F})$ containing it.
\end{definition}

\begin{remark}
In what is to follow, $T$ will also often denote a sloped Heegaard surface. Also note that if $S$ is a sloped Heegaard surface for $N$, then the entire boundary of $S$ lies in  $\boundary_L N$.
\end{remark}

\subsection{Sweepouts}

Suppose that $\ob{C}$ is a compressionbody containing the union $L$ of bridge arcs and vertical arcs in $\ob{C}$. Let $C = \ob{C} \setminus \inter{\eta}(L)$. We let $\boundary_\pm C = \boundary_\pm \ob{C} \setminus \inter{\eta}(L)$. A \defn{spine} for $(C,L)$ is an embedded graph $G\subset C$ defined as follows. $G$ is the union of a properly embedded graph $G_0 \subset C$ and a set of meridian loops $\mu$ on $\boundary \eta(L)$ such that the following hold:
\begin{itemize}
\item $G_0$ is disjoint from $\boundary_+ \ob{C}$.
\item If $\lambda$ is a vertical arc of $L$, then $\mu$ is disjoint from $\eta(\lambda)$.
\item If $\lambda$ is a vertical arc of $L$, then $G_0$ is disjoint from $\eta(\lambda)$.
\item If $\lambda$ is a bridge arc of $L$, then $G_0$ has a single valence-one vertex on $\boundary \eta(\lambda)$ and is otherwise disjoint from $\boundary \eta(\lambda)$.
\item If $\lambda$ is a bridge arc of $L$ then there is exactly one component of $\mu$ on $\boundary \eta(\lambda)$ and it contains the vertex of $G_0$ on that component.
\item The manifold $C \setminus \inter{\eta}(G_0)$ is homeomorphic to $\boundary_+ C \times I$ by a homeomorphism $\phi_C$ which is the identity on $\boundary_+ C = \boundary_+C \times \{1\}$.
\end{itemize}
See Figure \ref{fig:spine} for an example of a spine. We may extend $\phi_C$ to a map $\phi_C \co C \to [0,1]$ so that $\phi_C(\boundary_- C \cup G) = 0$ and $\phi_C(\boundary_+ C) = 1$ and for each $t \in (0,1]$, the surface $\phi_C^{-1}(t)$ is isotopic to $\boundary_+ C$. We call the map $\phi_C$ a \defn{half-sweepout}.

\begin{figure}[ht]
\labellist \small\hair 2pt
\pinlabel {$\boundary_+ C$} [b] at 177 156
\pinlabel {$\boundary_- C$} [b] at 177 10
\pinlabel {$L$} [l] at 43 69
\pinlabel {$L$} [t] at 354 102
\pinlabel {$\eta(L)$} [b] at 485 144
\pinlabel {$\eta(L)$} [b] at 774 144
\pinlabel {$G_0$} [l] at 556 61
\pinlabel {$G_0$} [l] at 810 51
\pinlabel {$\mu$} [t] at 747 80
\endlabellist
\centering
\includegraphics[scale=.4]{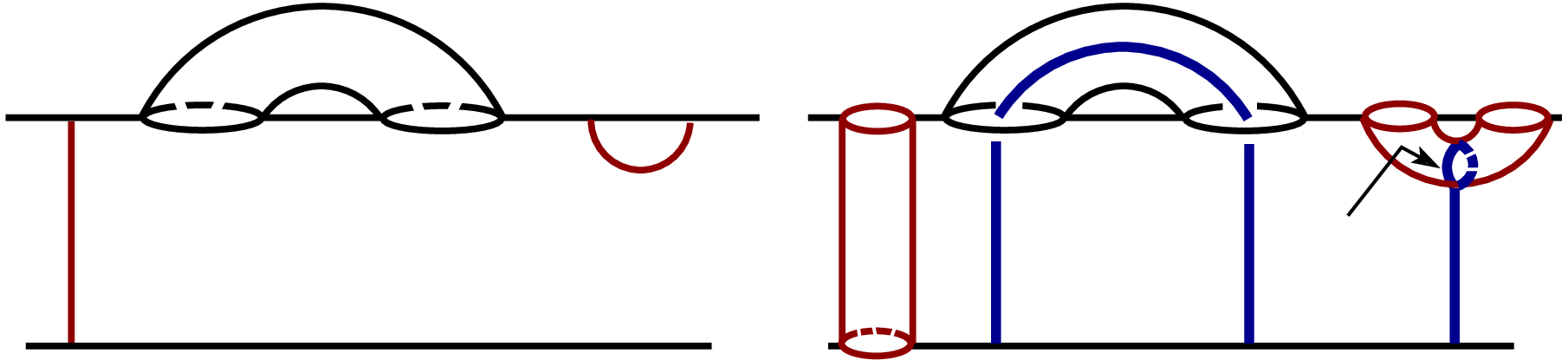}
\caption{On the left is an example of $(C,L)$ where the genus of $\boundary_+ C$ is one greater than the genus of $\boundary_- C$ and $L$ is the union of a vertical arc and a bridge arc. On the right, the spine is shown in thicker lines and, in the online version of the paper, in blue.}
\label{fig:spine}
\end{figure}
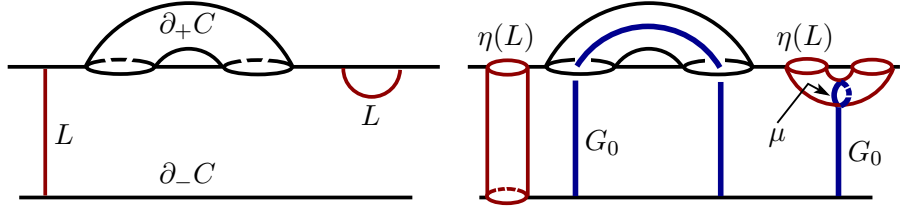

\begin{definition}
Suppose that $H$  is a sloped Heegaard surface for $M(L)$ inducing slope $\sigma$ on $\boundary_L M(L)$. Choose spines for $(H_\dn, L \cap H_\dn)$ and $(H_\up, L \cap H_\up)$. Let $\phi_{H_\dn}$ and $\phi_{H_\up}$ be the associated half-sweepouts. Define $\phi \co M(L) \to [0,1]$ as follows:
\[\phi(x) =
\begin{cases}
(\phi_{H_\dn}(x))/2 & \text{ if } x \in H_\dn \\
1 -\phi_{H_\up}(x)/2 & \text{ if } x \in H_\up \\
\end{cases}
\]
Observe that $\phi$ has the properties:
\begin{itemize}
\item  $\phi^{-1}(0)$ is the union of a spine for $(H_\dn, L \cap H_\dn)$ with $\boundary_- H_\dn$\
\item $\phi^{-1}(1)$ is the union of a spine for $(H_\up, L \cap H_\up)$ with $\boundary_- H_\up$\
\item For all $t \in (0,1)$, the surface $H_t = \phi^{-1}(t)$ is properly isotopic to $H$.
\end{itemize}
\end{definition}

We call the map $\phi$, and any small perturbation of $\phi$, a \defn{sweepout} of $M(L)$ associated to the surface $H$. To gain an informal understanding of what a sweepout is, observe that if we pick spines $G_\up$ and $G_\dn$ for the compressionbodies on either side of a bridge surface $\ob{H}$ for $(M,L)$, then $M(L) \setminus (G_\up \cup G_\dn)$ is homeomorphic to $H \times (0,1)$ and the sweepout is the projection onto the second factor. If $\mu$ is the union of the meridional curves on $\boundary \eta(L)$ which belong to $G_\up \cup G_\dn$, then we observe that $\boundary_L M(L) \setminus \mu$ corresponds to $\boundary H \times (0,1)$. For convenience, we extend the sweepout $\phi$ to be defined on all of $M(L)$ by sending $G_\dn$ to $0$ and $G_\up$ to 1.

We can perform a similar construction for thick surfaces of sloped generalized Heegaard surfaces. If $(\mc{F}, \mc{S})$ is a sloped generalized Heegaard surface for $M(L)$, a component $H \subset \mc{S}$ is a sloped Heegaard surface for the component $M_0$ of $M \setminus \inter{\eta}(\mc{F})$ containing it. We can construct a sweepout $\phi$ for $M_0$ as above and we call $\phi$ a \defn{partial sweepout} of $M(L)$. Each surface $H_t$ for $t \in (0,1)$ can be canonically identified with $H$ using the proper isotopy induced by the sweepout as noted above. When discussing isotopy classes of curves on $H_t$, we will often use this identification to think of them as isotopy classes of curves on $H$.

\subsection{Distance}\label{sec: dist}

The \textit{curve complex} $\mathcal{C}(T)$ of a compact surface $T \subset N$ is the graph whose vertices are isotopy classes of essential simple closed curves in $T$ and whose edges span pairs of isotopy classes of curves that have disjoint representatives.  The curve complex of a compact, connected, orientable surface is connected and infinite as long as the surface is not a genus 0 surface with four or fewer boundary components or a genus 1 surface with zero or one boundary components.  Let $\goodCurves$ be the set of homeomorphism classes of compact, connected, orientable surfaces with connected and infinite curve complex. If $T \in \goodCurves$, we make the vertex set of $\mathcal{C}(T)$ a metric space by defining the distance between two vertices as the number of edges in the shortest edge path between them.

If $T$ is a sloped Heegaard surface, the \defn{disc sets} $\mathcal{D}_{\up}$, $\mathcal{D}_{\dn}$ in $\mc{C}(T)$ are the sets of vertices representing loops that bound compressing discs for $T$ in $T_{\up}$ and $T_{\dn}$ respectively.

\begin{definition}
If $T$ is a sloped Heegaard surface and $T \in \goodCurves$, the \defn{distance} $d_{\mathcal{C}}(T)$ is the distance in $\mathcal{C}(T)$ between $\mathcal{D}_{\up}$ and $\mathcal{D}_{\dn}$. If $T \not\in \goodCurves$, for convenience, we set $d_\mc{C}(T) = -\infty$. If $\ob{T}$ is a bridge surface for $(M,L)$ we define $d_\mc{C}(\ob{T}) = d_\mc{C}(T)$.
\end{definition}

In other words, if $T \in \goodCurves$, the distance $d_{\mathcal{C}}(\ob{T})$ is the minimum $n$ such that there are essential simple closed curves $\gamma_0, \gamma_1, \hdots, \gamma_n$ in $\ob{T}$ (considered as a surface with marked points) such that $\gamma_0$ and $\gamma_n$ bound compressing discs for $T$ lying below and above $\ob{T}$ respectively and so that for all $i$, the curve $\gamma_i$ can be isotoped to be disjoint from $\gamma_{i+1}$.

\begin{remark}
If $\ob{T}$ is a four punctured sphere or unpunctured torus, the ``Farey complex'' is often used in place of the curve complex and a distance is defined in the Farey complex instead. Much of what we do should carry over to that setting, but for simplicity we consider only distance in the curve complex.
\end{remark}

\begin{definition}
Given $(M,L)$ where $M$ is a compact orientable 3-manifold and $L \subset M$ is a link, the \defn{distance} $d_{\mathcal{C}}(L)$ is the supremum of the set $\{d_\mc{C}(\ob{T})\}$ where $\ob{T}$ is a bridge surface for $(M,L)$ realizing $b(L)$ (recalling our convention that only minimal genus Heegaard surfaces in $M$ can realize $b(L)$).
\end{definition}

\begin{remark}
We observe that if $L \subset S^3$ is the unknot or a 2-bridge link, then $d_\mc{C}(L) = -\infty$, as $d_\mc{C}(T) = -\infty$ when $T$ is an annulus or planar surface with four boundary components. If $M \neq S^3$, or if $M= S^3$ and $b(L) \geq 3$, then $d_\mc{C}(L) \geq 0$. If fact, it follows from \cite{To07} that if $M$ is a closed, orientable, irreducible 3-manifold and if $L$ is a knot, then $d_\mc{C}(L) < \infty$ and the supremum is actually a maximum. That is, there is a bridge surface whose distance realizes $d_\mc{C}(L)$. To see this, let $\ob{T}'$ be any fixed bridge surface for $(M,L)$. Then for any other non-equivalent bridge surface $\ob{T}$, by \cite{To07} $d_{\mathcal{C}}(\ob{T}) \leq 2-\chi(T')$ where as usual $T'=\ob{T}'\cap M(L)$.
\end{remark}

The \defn{arc and curve complex} $\mathcal{AC}(T)$ for a compact, orientable surface $T$ is defined similarly; it is the simplicial complex whose vertices are isotopy classes of essential simple closed curves and essential properly embedded arcs. Edges span pairs of disjoint arcs/curves. The arc and curve complex is connected and infinite for any compact, connected, orientable surface which is not a planar surface with three or fewer boundary components or a genus 1 surface with no boundary components. We let $\goodArcs$ be the set of homeomorphism types of compact, connected, orientable surfaces which have infinite, connected arc and curve complex.

As with the curve complex, we make its vertex set a metric space by declaring each edge of $\mc{AC}(T)$ to have length one. Given a  bridge surface $\ob{T}$ of $(M,L)$ with $L \subset M$ a link, the collection of boundary compressing discs and compressing discs for $T$ intersects $T$ in essential arcs and essential loops, respectively. These arcs and loops define vertices in the arc and curve complex of $T$. The \defn{disc sets} of $\mc{AC}(T)$ are the sets $\mathcal{D}_{\up}$, $\mathcal{D}_{\dn}$ of vertices of $\mathcal{AC}(T)$ defined by the isotopy classes of arcs and circles in $T$ arising from the boundary compressing and compressing discs on the two sides of $T$.

\begin{definition}
If $T \in \goodArcs$ is a sloped Heegaard surface, the \defn{distance} $d_{\mathcal{AC}}(T)$ is the distance in $\mathcal{AC}(T)$ from the set $\mathcal{D}_{\up}$ to the set $\mathcal{D}_{\dn}$. If $T \not\in \goodArcs$, for convenience, we set $d_\mc{AC}(T) = -\infty$. If $\ob{T}$ is a bridge surface for $(M,L)$, we let $d_\mc{AC}(\ob{T}) = d_\mc{AC}(T)$. We define $d_\mc{AC}(L)$ to be the supremum of the set $\{d_\mc{AC}(\ob{T})\}$ for all bridge surfaces $\ob{T}$ for $(M,L)$ realizing $b(L)$.
\end{definition}

\begin{remark}
If $L \subset S^3$ in the unknot, then $d_\mc{AC}(L) = -\infty$; otherwise, for a link $L$ in a compact, orientable 3-manifold $M$, $d_\mc{AC}(L) \geq 0$. It follows from Lemma \ref{Lem: relating AC and C} below that if $M$ is closed, orientable and irreducible and if $L \subset M$ is a knot, then $d_\mc{AC}(L) < \infty$.
\end{remark}

In general, we can also define the distance between any two subsets of $\mc{C}(T)$ or $\mc{AC}(T)$, as follows:
\begin{definition}
If $T \in \goodCurves$ or if $T \in \goodArcs$, the \defn{distance} (in $\mc{C}(T)$ or $\mc{AC}(T)$ respectively) \defn{between two subsets} is just the minimal path distance between the subsets.
\end{definition}

The following is a well-known relationship between $d_{\mc{AC}}$ and $d_\mc{C}$.

\begin{lemma}\label{Lem: relating AC and C}
If $\ob{T}$ is a bridge surface with $T \in \goodCurves$, then \[d_{\mc{AC}}(T) \leq d_{\mathcal{C}}(T) \leq 2d_{\mathcal{AC}}(T).\]
\end{lemma}
\begin{proof}
Suppose that $T \in \goodCurves$. The first inequality is clear from the definitions. Consider a minimal length path $\alpha = \alpha_1, \alpha_2, \cdots, \alpha_n$ in $\mc{AC}(T)$ between disc sets in $\mc{AC}(T)$. It has length $n-1$. This path induces a path $\beta$ in $\mc{C}(T)$ of length at most twice the length of $\alpha$. Here is the construction: If $\alpha_i$ is represented by a loop, we let $\beta_{2i-1} = \alpha_i$. If $\alpha_i$ is represented by an arc, we let $\beta_{2i-1}$ be the loop which is a frontier in $\ob{T}$ of a regular neighborhood of an arc representing $\alpha_i$ (so that $\beta_{2i-1}$ bounds in $\ob{T}$ a disc with two marked points). If $\beta_{2i-1}$ and $\beta_{2i + 1}$ are isotopically disjoint, we let $\beta_{2i} = \beta_{2i-1}$. If $\beta_{2i-1}$ and $\beta_{2i + 1}$ are not isotopically disjoint, then $\alpha_i$ and $\alpha_{i+1}$ are represented by arcs in $T$ having disjoint interiors and sharing one or both endpoints. There is an essential loop $\beta_{2i}$ which is disjoint from the arc representing both $\alpha_i$ and $\alpha_{i+1}$. The path $\beta = \beta_1, \beta_2, \hdots, \beta_{2n-1}$ is then (after replacing a loop with its isotopy class) a path in $\mc{C}(T)$ of length at most twice the length of $\alpha$. Since $\alpha_1$ bounds either a bridge disc or a compressing disc on one side of $\ob{T}$, $\beta_1$ will bound a compressing disc on the same side of $\ob{T}$. Similarly, since $\alpha_n$ bounds a bridge disc or compressing disc on the other side of $\ob{T}$, the loop $\beta_{2n-1}$ will bound a compressing disc on that same side. Hence, $\beta$ is a path in $\mc{C}(T)$ between the disc sets in $\mc{C}(T)$ and so $d_\mc{C}(T) \leq 2d_\mc{AC}(T)$.
\end{proof}

In the proof of the previous lemma, we call $\beta$ a path in $\mc{C}(T)$ \defn{induced} by the path $\alpha$ in $\mc{AC}(T)$. See Figure~\ref{Fig: InducedPath} for an example.

\begin{figure}
\labellist \small\hair 2pt
\pinlabel {$1$} [b] at 102 225
\pinlabel {$2$} [b] at 167 225
\pinlabel {$3$} [b] at 336 190
\pinlabel {$4$} [t] at 336 244
\pinlabel {$1$} [b] at 101 92
\pinlabel {$2$} [b] at 122 119
\pinlabel {$3, 4$} [t] at 190 93
\pinlabel {$5$} [b] at 329 26
\pinlabel {$6$} [bl] at 392 129
\pinlabel {$7$} [t] at 332 128
\endlabellist
\centering
\includegraphics[scale=.6]{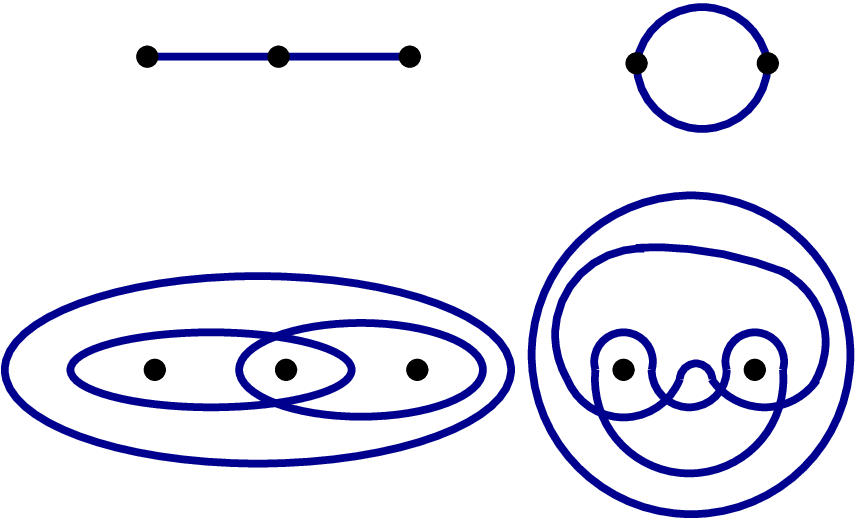}
\caption{The top row shows a path in $\mc{AC}$ and the second row shows the induced path in $\mc{C}$.}
\label{Fig: InducedPath}
\end{figure}
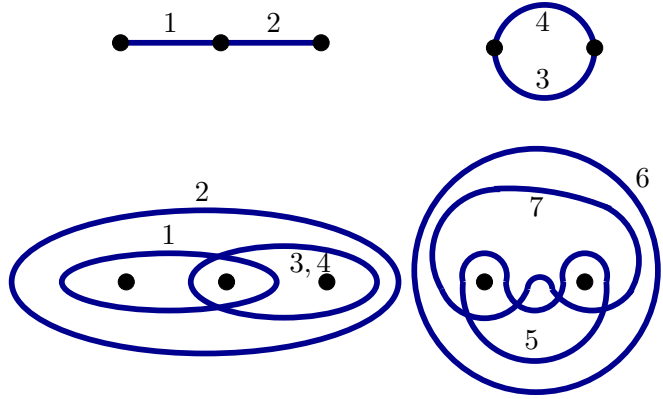

It follows from Lemma \ref{Lem: relating AC and C} and our previous remarks that if $L$ is a knot in a closed, orientable, irreducible 3-manifold $M$ then $d_\mc{AC}(L)$ is attained by $d_\mc{AC}(\ob{T})$ for some bridge surface $\ob{T}$ for $(M,L)$.

\section{Constructing paths}
\label{sec:essential surfaces}

As before, let $M$ be a compact, orientable manifold containing a properly embedded link $L$ and let $N=M(L)$. Let $S \subset N$ be a properly embedded compact orientable surface that is not necessarily connected. Let $\boundary_L N$ be the non-empty union of all components of $\boundary N \setminus \boundary M$ which have non-trivial intersection with $\boundary S$. (In this section, we do not require that $\boundary_L N = \boundary N \setminus \boundary M$.) Let $S$ have boundary slope $\sigma$ on $\boundary_L N$ and let $T$ be a thick surface of a sloped generalized Heegaard surface for $N$ such that $T$ has boundary slope $\tau$ on $\boundary_L N$. We assume $\Delta(\sigma,\tau) > 0$. We let $\boundary_L S$ denote the components of $\boundary S$ lying in $\boundary_L N$ and we let $\boundary_0 S=\boundary S \setminus \boundary_L S$. We define $\boundary_L T$ and $\boundary_0 T$ similarly for consistency but note that $\boundary T=\boundary_L T$ as $T$ is a thick surface of a sloped generalized Heegaard surface for $N$.




In this section, we prove a result which is key to our distance bounds. In order to apply the result in a number of different situations, we state it as generally as possible. We begin with two definitions.

The first definition is essentially that of  ``Morse position'' for the surface $S$ with respect to the sweepout by $T$. Given a sweepout $\phi \co M(L) \to [0,1]$, if $S \subset M(L)$ is a properly embedded surface, then we can perturb $\phi|_S\co S \to [0,1]$ to be a Morse function. Then for all regular values $t \in (0,1)$, $\phi^{-1}(t) \cap S$ is a properly embedded 1--manifold in $S$. For $t = \{0,1\}$, the set $\phi^{-1}(t) \cap S$ is the union of components of $\boundary S \cap \boundary M(L)$ with the points where $S$ intersects the spine $\phi^{-1}(t) \setminus \boundary_0 M(L)$. In addition to the requirement that $\phi|_S$ be Morse, we need a few other constraints, one of which concerns the possibility of $\phi|_S$ having more than one critical point at a particular height. When $S$ is an essential surface, we need not worry about this. However, when $S$ is a thick surface we will need to use a ``graphic'' argument. That argument requires us to consider the possibility that there may be two critical points at some height, as in the definition below.

\begin{definition}\label{Defn: adapted}
A (partial) sweepout $\phi$ of $N$ with level surfaces $T_t = \phi^{-1}(t)$ is \defn{adapted} to $S$ if all of the following hold:
\begin{enumerate}
\item $\boundary T_t$ intersects $\boundary S$ minimally for all $t \in (0,1)$.
\item The restriction $\phi|_S$ is Morse (i.e., has only non-degenerate critical points).
\item For all but at most one critical value $v$ in $(0,1)$, there is a unique critical point of $\phi|_S$ with critical value $v$.
\item There may be a single critical value $v \in (0,1)$, such that there are two critical points of $\phi|_S$ with critical value $v$ and there are no critical values corresponding to more than two critical points.
\item If $c_1$ and $c_2$ are distinct critical points of $\phi|_S$ with the same height $v$, then, for sufficiently small values of $\epsilon$, after canonically identifying $T_{v-\epsilon}$ and $T_{v+\epsilon}$ with $T$, every component of $T_{v-\epsilon} \cap S$ can be isotoped in $T$ to be disjoint from every component of $T_{v+\epsilon} \cap S$.
\end{enumerate}
\end{definition}

The next definition hones in on the important values of the sweepout which allow us to infer the topology of $S$ from the structure of the critical points. It is related to the idea of ``mutuality'' from \cite{BS}.
\begin{definition}
Suppose that $\phi$ is a sweepout of $N$ adapted to $S$ with level surfaces $T_t$. An interval $[a,b] \subset (0,1)$ with $a < b$ is \defn{essential} with respect to $\phi$ if all of the following hold:
\begin{itemize}
\item $a$ and $b$ are regular values of $\phi|_S$,
\item every arc of $T_t \cap S$ is essential in both $T_t$ and $S$ for all regular values $t \in [a,b]$
\end{itemize}
The interval $[a,b]$ is \defn{maximally essential} with respect to $\phi$ if it is essential and if  there is an interval $[u,v] \supset [a,b]$ with $u$ and $v$ regular values of $\phi|_S$ such that there is exactly one critical value of $\phi|_S$ in each of the intervals $[u,a]$ and $[b,v]$ and, additionally, there is an arc or circle of $T_u \cap S$ that is essential in $T_u$ but bounds a disc in $(T_u)_\dn$ and there is an arc or circle of $T_v \cap S$ that is essential in $T_v$ but bounds a disc in $(T_v)_\up$. See Figure~\ref{Fig: EssInt}. When the sweepout is obvious from context, we will simply say that $[a,b]$ is essential or maximally essential.
\end{definition}

\begin{figure}
\labellist \small\hair 2pt
\pinlabel {$v$} [l] at 299 344
\pinlabel {$b$} [l] at 299 291
\pinlabel {$u$} [l] at 299 23
\pinlabel {$a$} [l] at 300 92
\endlabellist
\centering
\includegraphics[scale=.4]{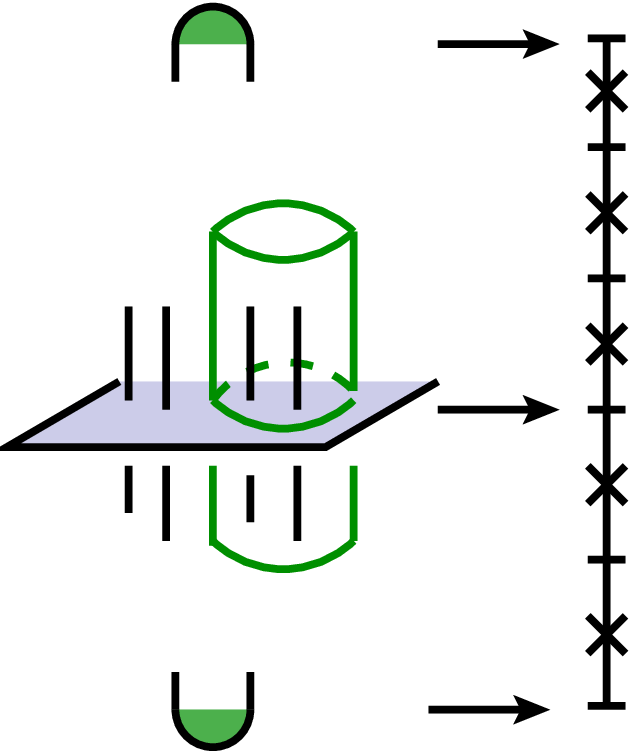}
\caption{A schematic depiction of maximally essential interval. The ``x''s represent the critical values of $\phi|_S$.}
\label{Fig: EssInt}
\end{figure}
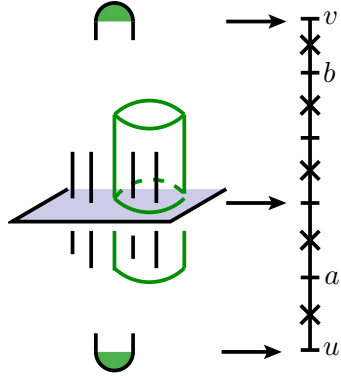

Here is the theorem which is key to our results:

\begin{theorem}\label{thm:distanceL}
Let $\phi$ be a  sweepout of $N$ by $T$ adapted to $S$ with maximally essential interval $[a,b]$. Assume that no component of $S$ is a sphere or disc. Then:
\begin{equation}\label{Ineq: Arc Bound}
\Delta b_{\min}(T)(d_{\mathcal{AC}}(T) - 2) \leq \left(\frac{4g(S) -4|S| + 2|\boundary_0 S|}{|\boundary_L S|}\right) + 2
\end{equation}
and
\begin{equation}\label{Ineq: Curve Bound}
\Delta b_{\min}(T)(d_{\mathcal{C}}(T) - 4) \leq \left(\frac{8g(S) -8|S| + 4|\boundary_0 S|}{|\boundary_L S|}\right) + 4,
\end{equation}
where $\Delta$ denotes the intersection number of the multislopes represented by $\partial_L S$ and $\partial_L T$.
\end{theorem}

In addition to the assumptions established at the start of this section, for the remainder of this section we also assume that $\phi$ is a sweepout by $T$ adapted to $S$ and with maximally essential interval $[a,b]$ and that no component of $S$ is a disc or sphere.

The basic outline of the proof of Theorem \ref{thm:distanceL} is as follows: Let $T_t$ be the level surface of $\phi$ at height $t \in (0,1)$. Since $\phi$ is adapted to $S$, as $t$ varies from $a$ to $b$, the intersection $T_t \cap S$ is a collection of arcs and circles whenever $t$ is a regular value of $\phi|_S$. For heuristic purposes, assume that all these arcs and circles are essential in both $S$ and $T$ (this is close to being the case since the interval is essential). Then each critical point of $\phi|_S$ is a saddle singularity on $S$. These saddles are ``essential'' and so contribute to $-\chi(S)$ (as observed by Bachman and Schleimer \cite{BS}).  This observation is formalized in Lemma~\ref{lem: lower bound on neg euler}.

As $T_t$ passes through a saddle singularity of $S$, the isotopy classes in $T$ of the intersection arcs may change. We show that the total number of times we have such a change is at least the product of the length of a certain path in $\mc{AC}(T)$ with quantities involving the bridge number and $|\boundary_L S|$. This is formalized in Lemma~\ref{Lem: Path lemma}. The length of the path is related in a straightforward way to $d_\mc{AC}(T)$ and $d_\mc{C}(T)$. The fact that $|\boundary_L S|$ appears in the inequality allows us to obtain an inequality involving just the genus of $S$ and not merely its euler characteristic.

Before tackling the two parts of the proof, we make a few remarks and definitions.

If $T \not\in \goodArcs$, then $d_\mc{AC}(T)=-\infty$ and Inequality \eqref{Ineq: Arc Bound} holds trivially. Henceforth, we assume that if we are considering $\mc{AC}(T)$, then $T \in \goodArcs$. Similarly, if $T \not\in \goodCurves$, then $d_\mc{C}(T) = - \infty$ and  Inequality \eqref{Ineq: Curve Bound} holds trivially. Henceforth, we assume that if we are considering $\mc{C}(T)$, then $T \in \goodCurves$.

Recall that, for each regular value $t$, $|\boundary_L T_t \cap \boundary_L S| = |\boundary_L T \cap \boundary_L S|$ and there is a canonical identification of the points of $\boundary_L T_t \cap \boundary_L S$ with the points of $\boundary_L T \cap \boundary_L S$. Henceforth, we do not distinguish between the points in $\boundary_L T_t \cap \boundary_L S$ and the points in $\boundary_L T \cap \boundary_L S$.

Let $v_1 < v_2 < \hdots < v_k$ be the critical values of $\phi|_S$ in $[a,b]$ and let $a = t_0 < t_1 < \hdots < t_k = b$ be regular values of $\phi|_S$ so that $v_i$ is the unique critical value of $\phi|_S$ in the interval $[t_{i-1}, t_i]$. Define a point $\lambda \in \boundary_L T \cap \boundary_L S$ to be \defn{active} at $v_i$ if as $t$ varies from $t_{i-1}$ to $t_i$ the isotopy class of an arc in $S$ with an endpoint $\lambda$ changes. The arc of $T_{t_{i-1}} \cap S$ with an endpoint $\lambda$ is said to be \defn{pre-active} at $v_i$ and the arc of $T_{t_i} \cap S$ with an endpoint $\lambda$ is said to be \defn{post-active} at $v_i$. An arc is \defn{active} if it is pre-active or post-active at some $v_i$. Let $Q$ be the total number of post-active arcs in $S$ (where the count is taken over all critical values $v_i$).  The number $Q$ is also half the total number of times that points in $\boundary_L T \cap \boundary_L S$ are active (observing that a point may be active more than once). A critical value $v_i$ is \defn{active} if there is an active arc at $v_i$. A critical point $p$ is \defn{active} if $\phi(p)$ is active. We note that, by the definition of essential interval, the pre-image of an active critical value contains an index 1 critical point. Let $\mc{C}$ be the set of active critical points.

Passing through the critical points at height $v_i$ results in some arcs and circles of $T_{t_{i-1}} \cap S$ being banded to other arcs and circles. The result of the banding are the arcs and circles of $T_{t_i} \cap S$. Define a circle component of $T_{t_{i-1}} \cap S$ to be \defn{pre-active} at $v_{i}$ if it is banded to a pre-active arc at $v_{i}$. Define it to be \defn{post-active} at $v_i$ if it is banded to a post-active arc at $v_i$. A circle is \defn{active} if it is pre-active or post-active at some $v_i$. By the definition of ``pre-active'' arc and ``post-active'' arc, each active circle is essential in $S$. See Figure~\ref{fig:saddles} for examples of active critical points and active circles.

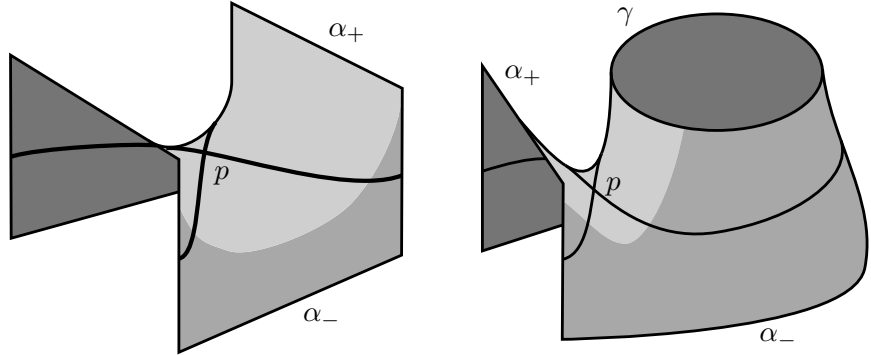
\begin{figure}[htb]
\labellist
\small\hair 2pt
\pinlabel $p$ [tl] at 143 136
\pinlabel $p$ [l] at 420 119
\pinlabel $\alpha_-$ [tl] at 206 35
\pinlabel $\alpha_-$ [tl] at 530 21
\pinlabel $\alpha_+$ [bl] at 224 219
\pinlabel $\alpha_+$ [bl] at 349 189
\pinlabel $\gamma$ [br] at 446 231
\endlabellist
\centering
  \includegraphics[width=4.5in]{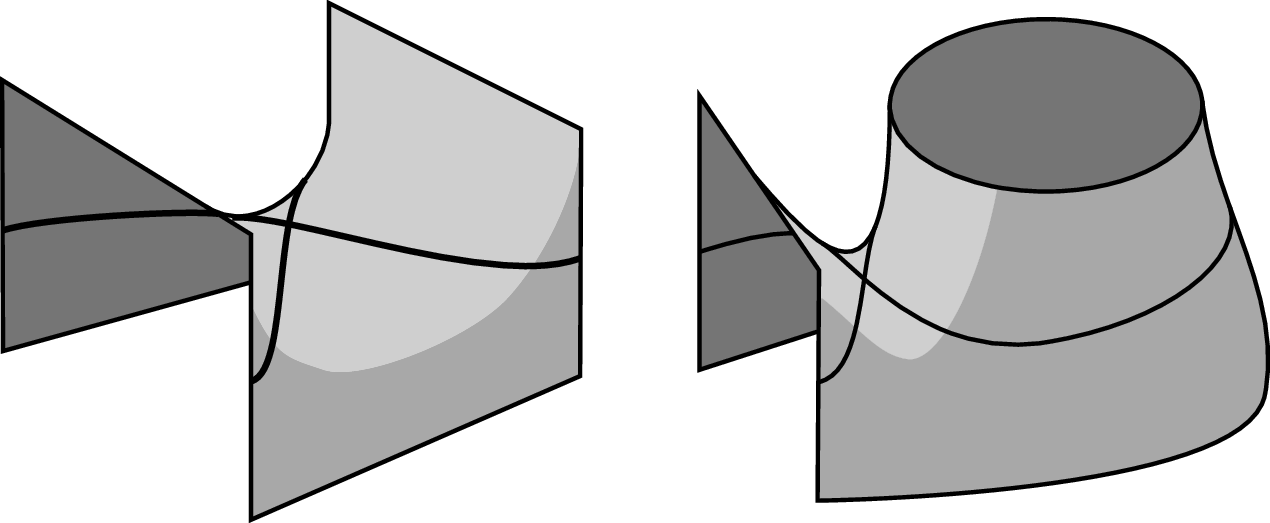}
  \caption{Two examples of index 1 critical points on $S$. In each example, the vertical line segments correspond to portions of $L$. The top and bottom of each portion of $S$ are cut off by portions of $T_{t_i}$ and $T_{t_{i-1}}$ respectively. In each example, if the arcs $\alpha_-$ and $\alpha_+$ are not isotopic, then $\alpha_-$ is pre-active at $\phi(p)$ and $\alpha_+$ is post-active at $\phi(p)$. On the right, if the arcs $\alpha_-$ and $\alpha_+$ are not isotopic, then the circle $\gamma$ is also post-active.}
  \label{fig:saddles}
\end{figure}

Suppose that $c_1$ and $c_2$ are distinct index 1 critical points both having height $v_i$. As $t$ varies from $t_{i-1}$ to $t_i$, two disjoint bands are attached to arcs and circles of $T_{t_{i-1}} \cap S$ to produce $T_{t_i} \cap S$. Say that $c_1$ and $c_2$ are \defn{connected} if  the union of the pre-active arcs and circles with the two bands is connected. See Figure~\ref{fig:connectedcp}.

\begin{figure}[htb]
\labellist \small\hair 2pt
\pinlabel {$\boundary T$} at 26 340
\pinlabel {$\boundary T$} at 26 219
\pinlabel {$\boundary T$} at 26 110
\pinlabel {$\boundary T$} at 217 340
\pinlabel {$\boundary T$} at 217 219
\pinlabel {$\boundary T$} at 217 110
\endlabellist
\centering
\includegraphics[scale=.38]{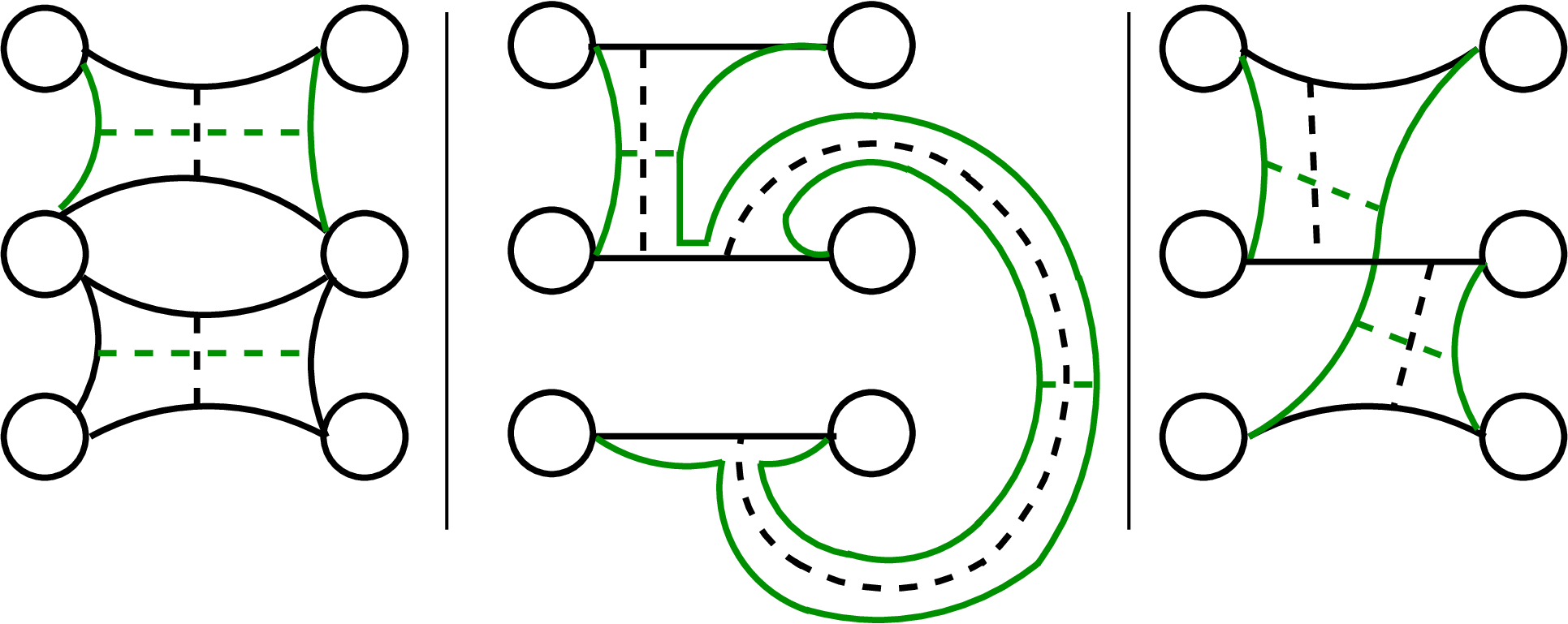}
\caption{Three examples of distinct critical points at the same height. The solid arcs which are horizontal and, in the online version, black arise from the intersection $T_{t_{i-1}} \cap S$. The solid arcs which are not horizontal and, in the online version, are green arise from the intersection $T_{t_i} \cap S$. The dashed lines represent the bands which we attach to move from the arcs of intersection before the critical value to the arcs of intersection after the critical value. The critical points lie at the intersection of the dashed lines. The critical points on the left are not connected and do satisfy condition (5) of Definition \ref{Defn: adapted}. The critical points in the middle example are connected and also satisfy condition (5) of Definition \ref{Defn: adapted}. The critical points on the right are connected and do not satisfy condition (5) of Definition \ref{Defn: adapted}.}
\label{fig:connectedcp}
\end{figure}

We now tackle the first part of the proof of Theorem~\ref{thm:distanceL}.
\begin{lemma}\label{lem: lower bound on neg euler} We have:
\[Q \leq -2\chi(S) \] where $Q$ is the total number of post-active arcs.
\end{lemma}

In the proof of Lemma \ref{lem: lower bound on neg euler}, we relate $Q$, which is the number of post-active arcs, to the euler characteristic of $S$ by showing that active critical points contribute in a lasting way to $\chi(S)$. We need to focus on active critical points, as  a non-active index 1 critical point may have its contribution to $\chi(S)$ cancelled out by an index 0 critical point, for example.

\begin{proof}
If $c$ is the unique index 1 critical point with height $v_i \in [a,b]$, then there are at most 2 post-active arcs in $T_{t_i} \cap S$. If $c_1$ and $c_2$ are index 1 critical points both having height $v_i$, then there are at most 4 post-active arcs in $T_{t_i} \cap S$. (If, as $t$ passes through $v_i$, the two bands are attached to four distinct arc components of $T_{t_{i-1}} \cap S$ then there are four post-active arcs at $v_i$. If any of the four ends of the two bands are adjacent to the same component of $T_{t_{i-1}} \cap S$, then there will be fewer than 4 post-active arcs.) In $S$, the post-active arcs at $T_{t_i} \cap S$ are (pairwise) disjoint from the post-active arcs of $T_{t_{j}} \cap S$, for $j \neq i$, since $T_{t_i} \cap T_{t_j} = \nil$. Thus, we have:

\textbf{Observation 1:} $Q \leq 2|\mc{C}|$ where $\mc{C}$ is the set of active critical points.

Let $P$ be the complement of an open regular neighborhood of the active arcs and circles in $S$. Denote its components by $P_1, \hdots, P_m$. The boundary of each $P_k$ is the union of copies of active arcs and circles as well as arcs or circles lying in $\boundary S \subset \boundary N$. Each $P_k$ contains at most 2 active index 1 critical points of $\phi|_S$, and there is at most one $P_k$ that contains 2 active index 1 critical points. Let $b_k$ be the number of copies of active arcs in $\boundary P_k$ and define the \defn{index} of $P_k$ to be:
 \[
 J(P_k) = (b_k/2) - \chi(P_k).
 \]
Since each active arc shows up twice in $\boundary P$ and since euler characteristic increases by one when cutting along an arc, we have:

\textbf{Observation 2:} $-\chi(S) = \sum_k J(P_k)$

Fix $k$. Since no component of $S$ is a sphere or disc, $P_k$ is not a sphere or disc disjoint from $\boundary N$.  If $P_k$ is a disc, its boundary cannot be an active circle or contain only a single active arc as active circles and arcs are essential in $S$. Thus, if $P_k$ is a disc, $J(P_k) \geq 0$. If $P_k$ is not a disc, $-\chi(P_k) \geq 0$. Therefore $J(P_k) \geq 0$, for all $k$. In particular, $J(P_k) \geq 0$ if $P_k$ does not contain an active critical point.

Suppose that $\alpha$ is a pre-active arc at an active critical value $v_i$. The post-active arcs at $v_i$ are obtained by either banding $\alpha$ to another pre-active arc $\beta$ at $v_i$, in which case we say that $\alpha$ is \defn{paired}, or by banding $\alpha$ to itself or to a circle component of $T_{t_{i-1}} \cap S$, in which case we say that $\alpha$ is \defn{solitary}.

Suppose $P_k$ contains a unique active critical point $c \in P_k$ and let $\alpha$ be a pre-active arc at $c$. If $\alpha$ is a paired arc at $c$, then $b_k \geq 4$ since there must be at least two pre-active arcs and two post-active arcs. Then $J(P_k) \geq (4/2) - 1 = 1$. If $\alpha$ is a solitary pre-active arc at $c \in \mc{C} \cap P_k$, then let $\gamma$ be the pre-active circle that is either banded to $\alpha$ at $c$ or that results from banding $\alpha$ to itself at $c$. In this case, $\gamma$ is essential in $S$ so $P_k$ is not a disc and there are two active arcs in the boundary of $P_k$ so $J(P_k) \geq (2/2) - 0 = 1$. Hence, if $P_k$ contains a single active index 1 critical point, then $J(P_k) \geq 1$.

The analysis when $P_k$ contains two active index 1 critical points is similar, but somewhat more delicate. In this case, the index 1 critical points must have the same height $v$ and be connected.  Note that all post-active arcs and circles at $v$ and all pre-active arcs and circles at $v$ are contained in $\partial P_k$. Furthermore, every component $P_k$ that contains a critical point contains at least two active arcs in its boundary. Additionally, the bands must lie in $P_k$ (since the bands, themselves, contain the index 1 critical points). Let $C_-$ and $C_+$ denote the number of pre-active and post-active circles at $v$, respectively. Then if $P_k$ is a disc, we have that $C_- = C_+ = 0$ and $b_k$ must be 6, in which case $J(P_k) = 2$.  Suppose that $P_k$ is not a disc. Then either $C_- + C_+ \geq 2$ and $b_k \geq 2$; $C_- + C_+ \geq 1$ and $b_k \geq 4$; or $b_k = 2$ and $P_k$ is not planar. In any of these cases, $J(P_k) \geq 2$. Consequently:

\textbf{Observation 3:} $\sum J(P_k) \geq |\mc{C}|$.

Combining Observations 1, 2, and 3 completes the proof.
\end{proof}

We now show how to relate $Q$ to paths in $\mc{AC}(T)$. Before beginning, we take note of the following arithmetic:  Summing over the components $V$ of $\boundary_L N$ we have, by the definition of $\Delta$,
\[\begin{array}{rcl}
|\boundary T \cap \boundary S| &\geq& \sum_V |\boundary T \cap V| |\boundary S \cap V|\Delta \\
&\geq& \sum_V 2b_{\min}(T)|\boundary S \cap V|\Delta \\
&=& 2b_{\min}(T)|\boundary_L S|\Delta
\end{array}
\]
\begin{lemma}\label{Lem: Path lemma}
There is a path
\[\alpha = \alpha_0, \alpha_1, \hdots, \alpha_\ell \]
 in $\mc{AC}(T)$ from the set of vertices represented by the arcs of $T_a \cap S$ to the set of vertices represented by the arcs of $T_b \cap S$ such that:
\begin{enumerate}
\item Each $\alpha_i$ is represented by an arc of $T_{t_i} \cap S$.
\item Letting the length of $\alpha$ be $\ell$, we have
\[
\ell b_{\min}(T)|\boundary_L S|\Delta \leq Q
\]
\end{enumerate}
\end{lemma}
\begin{proof}
Since arcs have two endpoints, the number $Q$ is equal to half the total number of times the points in $\boundary_L T \cap \boundary_L S$ are active. For a point $\lambda \in \boundary_L T \cap \boundary_L S$ let $n_\lambda$ denote the number of times it is active. Thus,
\[
\frac{1}{2} \sum_{\lambda \in \boundary_L T \cap \boundary_L S} n_\lambda = Q
\]
Let $\lambda_0$ be a point in $\boundary_L T \cap \boundary_L S$ which minimizes $n_\lambda$. As we saw above, the total number of points in $\boundary_L T \cap \boundary_L S$ is greater than or equal to $2b_{\min}(T)|\boundary_L S|\Delta$

Thus,
\[
n_{\lambda_0}b_{\min}(T)|\boundary_L S|\Delta \leq Q.
\]

Let $\alpha_i$ be the arc in $T_{t_i} \cap S$ with endpoint $\lambda_0$. Then (after taking isotopy classes in $T$) the sequence
\[
\alpha_0, \alpha_1, \hdots, \alpha_n
\]
is a sequence of vertices in $\mc{AC}(T)$. For each $i \in \{0, \hdots, n-1\}$, the vertex $\alpha_i$ is either identical to the vertex $\alpha_{i+1}$ or $\alpha_i$ and $\alpha_{i+1}$ are the endpoints of an edge in $\mc{AC}(T)$. Eliminating repetitions and renumbering, we have our desired path in $\mc{AC}(T)$. It has length $\ell = n_{\lambda_0}$. Thus we have conclusions (1) and (2).
\end{proof}

We can now complete the proof of Theorem~\ref{thm:distanceL}.

\begin{proof}[Proof of Theorem~\ref{thm:distanceL}]

By Lemma~\ref{lem: lower bound on neg euler}, $Q \leq -2\chi(S)$. The interval $[a,b]$ is maximal, so there is an interval $[u,v] \supset [a,b]$ such that $u,v$ are regular values of $\phi|_S$ and so that there is a unique critical value in each of the intervals $(u,a)$ and $(b,v)$. Furthermore, some component $\gamma_u$ of $T_u \cap S$ is essential in $T_u$ and bounds a disc below $T_u$ and some component $\gamma_v$ of $T_v \cap S$ is essential in $T_v$ and bounds a disc above $T_v$. Each component of $T_a \cap S$ is distance (in $\mc{AC}(T)$) at most 1 from each component of $T_u \cap S$ and each component of $T_b \cap S$ is distance at most 1 from each component of $T_v \cap S$.

Let $\alpha = \alpha_0, \hdots, \alpha_\ell$ be the path in $\mc{AC}(T)$ provided by Lemma~\ref{Lem: Path lemma}. Since $d_{\mc{AC}}(\alpha_0, \gamma_u) \leq 1$ and $d_{\mc{AC}}(\alpha_\ell, \gamma_v) \leq 1$, we have $\ell \geq d_{\mc{AC}}(T) - 2$.

Consequently, by Lemma~\ref{Lem: Path lemma}, we have
\[
(d_\mc{AC}(T) - 2)b_{\min}(T)|\boundary_L S| \Delta \leq -2\chi(S).
\]
Using the fact that \[-2\chi(S) = 4g(S) - 4|S| + 2|\boundary_0 S| + 2|\boundary_L S|,\] we obtain Conclusion (1).

Conclusion (2) follows immediately by applying Lemma \ref{Lem: relating AC and C} to Conclusion (1).
\end{proof}

For convenience we observe that a stronger version of the inequalities also hold when $T$ is weakly $\boundary$-reducible, even in the absence of a maximally essential interval.

\begin{lemma}\label{Lem: T weak red}
Suppose that $T$ is weakly $\boundary$-reducible. Assume that no component of $S$ is a sphere or disc. Then:
\[
\Delta b(T)(d_{\mathcal{AC}}(T) - 2) \leq \frac{4g(S) - 4|S| + 2|\boundary_0 S|}{|\boundary_L S|} + 2
\]
and
\[
\Delta b(T)(d_{\mc{C}}(T) - 4) \leq \frac{8g(S) - 8|S| + 4|\boundary_0 S|}{|\boundary_L S|} + 4
\]
\end{lemma}
\begin{proof}
Since $T$ is weakly $\boundary$-reducible, $d_{\mc{AC}}(T) \leq 1$ and, by Lemma~\ref{Lem: relating AC and C}, $d_{\mc{C}}(T) \leq 2$.

Suppose that the first conclusion does not hold. Then
\[
-2 > \frac{4g(S) - 4|S| + 2|\boundary_0 S|}{|\boundary_L S|}
\]
Thus,
\[
0 > 4g(S) - 4|S| + 2|\boundary_0 S| + 2|\boundary_L S|
\]
There must be some component $S'$ of $S$ for which
\[
4 > 4g(S') + 2|\boundary_0 S'| + 2|\boundary_L S'|.
\]
We must have $g(S') = 0$, and
\[
3 \geq  2(|\boundary_0 S'| + |\boundary_L S'|).
\]
Since $S'$ is not a sphere or disc, the right hand side is at least 4, which is a contradiction.

Again we can obtain the second inequality by applying Lemma \ref{Lem: relating AC and C} to the first one.
\end{proof}

By a slight variation of the argument in the proof of Theorem \ref{thm:distanceL}, we can give an alternative upper bound on bridge number of knots as follows.  This will give us our best results on reducing and toroidal surgeries.

\begin{theorem}\label{Bounding Bridge}
Let $\ob{T}$ be a bridge surface for a knot $L$ in a compact orientable manifold $M$ and let $S$ be a properly embedded surface in $N$, no component of which is a disc or sphere. Assume that the slopes induced by $T$ and $S$ have $\Delta > 0$. Assume, also, that $d_\mc{C}(T) \geq 3$, and that there is a maximally essential interval for $S$ relative to $T$. Then,
\[
\Delta(b(T) - 4) \leq \frac{4 g(S) - 4|S| +2 |\boundary_0 S|}{|\boundary_L S|} + 2.
\]
\end{theorem}

\begin{proof}
Recall that $Q$ is the number of post-active arcs, and that by Lemma~\ref{lem: lower bound on neg euler},
\begin{equation}\label{fundamental ineq}
Q \leq -2\chi(S).
\end{equation}

Let $[a,b]\subset [u,v]$ be a maximally essential interval for $S$ relative to $T$. Recall that
\[ v_1 < v_2 < \hdots < v_{n-1}\]
are the critical values of $\phi|_S$ in $(a, b)$ and set $t_0 = a$ and $t_n = b$. Let $t_i = (v_i + v_{i-1})/2$. Let $\alpha$ be the arc or loop in $T_{u} \cap S$ that bounds a compressing or boundary compressing disc for $T_u$ that lies in $\ob{T}_\dn$ and let $\beta$ be the arc or loop in $T_{v} \cap S$ that bounds a compressing or boundary compressing disc for $T_v$ that lies in $\ob{T}_\up$. (In particular $\alpha$ and $\beta$ are essential in $T_u$ and $T_v$ respectively.)

A \defn{constant path} is a sequence of inactive arcs $\kappa_1, \hdots, \kappa_{n}$ with $\kappa_i \subset T_{t_i} \cap S$ and all $\kappa_i$ mutually isotopic in both $S$ and $T$.

Suppose that $(\kappa_i)$ is a constant path and canonically identify each $T_{t}$ with $T$.  Recall that the interior of $\kappa_1$ is disjoint from $\alpha$, the interior of $\kappa_n$ is disjoint from $\beta$, and $\kappa_1$ and $\kappa_n$ are isotopic (relative to their endpoints) in $T$. In particular the isotopy classes of $\alpha$, $\kappa_1$, and $\beta$ are a path of length 2 in $\mc{AC}(T)$. If no component of $\boundary T$ which is incident to $\kappa_1$ is also incident to $\alpha$ or $\beta$, then the induced path $\gamma_0$, $\gamma_1$, and $\gamma_2$ in $\mc{C}(T)$, has length 2, as in Figure \ref{Fig: path}.  However, $\gamma_0 \in \mc{D}_{\dn}$ and $\gamma_2 \in \mc{D}_\up$, so $d_\mc{C}(T) \leq 2$, contradicting the hypotheses of the theorem. Consequently, whenever $(\kappa_i)$ is a constant path, one endpoint of $\kappa_1$ lies in a component of $\boundary T$ adjacent to $\alpha$ or $\beta$.

\begin{figure}[ht]
\labellist \small\hair 2pt
\pinlabel {$\alpha$} [b] at 356 490
\pinlabel {$\beta$} [t] at 356 52
\pinlabel {$\kappa_1$} [l] at 306 312
\pinlabel {$\gamma_0$} [bl] at 477 517
\pinlabel {$\gamma_1$} [t] at 378 293
\pinlabel {$\gamma_2$} [br] at 64 499
\endlabellist
\centering
\includegraphics[scale=0.5]{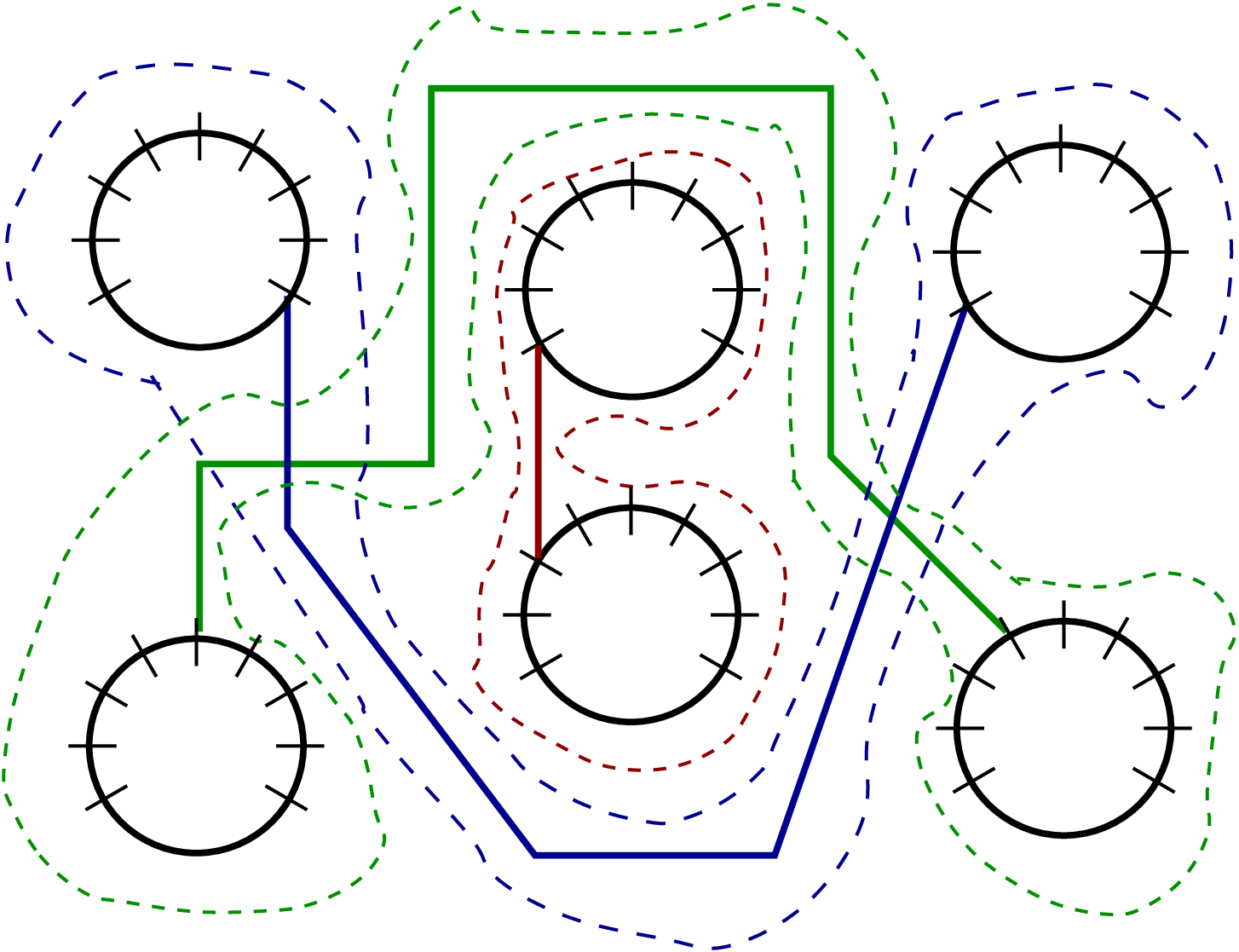}
\caption{The loops $\gamma_0$, $\gamma_1$, and $\gamma_2$ enclosing $\alpha$, $\beta$, and $\kappa_1$ (respectively) and the boundary components of $T$ incident to their endpoints form a path of length 2 in $\mc{C}(T)$.}
\label{Fig: path}
\end{figure}

Recall that a point of  $\boundary_L T\cap \boundary_L S$  is \defn{active} if the adjacent arc of $T\cap S$ changes isotopy class in $T$ at some critical value of $\phi|_S$. A point is \defn{inactive} if it is not active at any critical value. Each inactive point is an endpoint of an arc in a constant path. Since $d_\mathcal{C}(T) \geq 3$, each arc in a constant path is incident to one of the components of $\boundary_L T$ incident to either $\alpha$ or $\beta$. There are at most 4 components of $\boundary_L T$ incident to $\alpha \cup \beta$, so there are at most 8 components of $\boundary_L T$ incident to an arc in a constant path. Because $L$ is a knot, each component of $\boundary_L T$ has exactly $|\boundary_L S|\Delta$ points of intersection with $\boundary_L S$. Hence, there are at most $8|\boundary_L S|\Delta$ inactive points in $\boundary_L T \cap \boundary_L S$. There are exactly $2b(T)|\boundary_L S|\Delta$ points in $\boundary_L T\cap \boundary_L S$, so there are at least $(2b(T) - 8)|\boundary_L S|\Delta$ active points. Each active arc is incident to two active points. Thus, by Inequality \eqref{fundamental ineq},
\[
(b(T) - 4)|\boundary_L S|\Delta \leq Q \leq -2\chi(S).
\]
We have $-2\chi(S) = 4g(S) - 4|S| +2 |\boundary_0 S|+ 2|\boundary_L S|$.
Consequently,
\[
(b(T) - 4)\Delta \leq \big(4g(S) - 4|S| +2 |\boundary_0 S|\big)/|\boundary_L S| + 2.
\]\end{proof}

\section{Bounds for essential surfaces}\label{Non-MeridoinalEss}
In this section, assume that $S$ is an essential, compact, orientable surface properly embedded in $N=M(L)$, with $L$ a link and $N$ irreducible and $\boundary$-irreducible. We show that the genus of $S$ produces an upper bound on the product of the bridge number and distance of $T$, a sloped Heegaard surface for $N$. As usual, let $\sigma$ be the multislope represented by $\boundary_L S$ and $\tau$ the multislope represented by $T$. Assume, as usual, that $\Delta(\sigma,\tau) \geq 1$.

We begin by showing that there is a maximally essential interval for a sweepout of $N$ by $T$, adapted to $S$. The proof is standard and can, in essence, be found in Gabai's original use of thin position \cite{G}.

\begin{lemma}\label{Lem: ess surface}
Suppose $T$ is a strongly $\boundary$-irreducible sloped Heegaard surface for $N$ and $S$ is essential. Then there is a sweepout $\phi$ by $T$ adapted to $S$ such that each critical value of $\phi|_S$ has a unique critical point in its pre-image and so that there is a maximally essential interval with respect to $\phi$.
\end{lemma}

\begin{proof}
Let $\phi$ be a sweepout of $N$ by $T$. By standard Morse theory arguments and by choosing a flat metric on each component of $\boundary_L N$ and requiring that the boundaries of the level sets of $\phi$ be geodesics, we may perturb $\phi$ so that $\phi$ is adapted to $S$ and so that every critical value of $\phi|_S$ has a unique critical point in its preimage. Let \[ v_0 = 0 < v_1 <  \hdots < v_{n-1}< 1 = v_n\] be the critical values of $\phi|_S$ and let $I_i = (v_{i-1}, v_i)$. Label each of these intervals with $\up$ (respectively, $\dn$) if, for $t \in I_i$, some component of $T_t \cap S$ is an essential loop in $T_t$ that bounds a compressing disc above (respectively, below) $T_t$ or some component of $T_t \cap S$ is an essential arc in $T_t$ that bounds a $\partial$-compressing disc above (respectively, below) $T_t$. The label is independent of the choice of $t$ in the interval. Since $T$ is strongly $\boundary$-irreducible, no interval is labelled both $\up$ and $\dn$.

Suppose adjacent intervals $I_i$ and $I_{i+1}$ have distinct labels. Let $t_i\in I_i$ and $t_{i+1}\in I_{i+1}$. Since every critical value of $\phi|_S$ has a unique critical point in its preimage, then components of $T_{i+1}\cap S$ are obtained from the components of $T_i \cap S$ by inserting a nullhomotopic loop (if there is a minimum at $v_i$), deleting a nullhomotopic loop (if there is a maximum at $v_i$), or by banding two (not necessarily distinct) components of $T_i \cap S$ together (if there is a saddle at $v_i$). In each of these cases every component of $T_i \cap S$ can be isotoped to be disjoint from every component of $T_{i+1} \cap S$. Since $I_i$ and $I_{i+1}$ have distinct labels, this implies that $T$ is weakly $\boundary$-reducible, a contradiction to our assumption that $T$ is strongly $\boundary$-irreducible. Hence, no two adjacent intervals have distinct labels.

When $t$ is close to 1, since $L$ is a link in bridge position with respect to $\ob{T}$, every arc of $T_t \cap S$ is essential in $T_t$ and bounds a $\boundary$-compressing disc above $T_t$. Thus, $I_n$ is labelled $\up$. Similarly, $I_1$ is labelled $\dn$. Thus there are $i$ and $i'$ with $i < i'-1$ so that $I_i$ is labelled $\dn$, $I_{i'}$ is labelled $\up$, and for each $j$ with $ i < j < i'$, the interval $I_j$ is unlabelled. Choose $u \in I_i$, $v \in I_{i'}$ and $a \in I_{i+1}$, $b \in I_{i' - 1}$ with $a < b$. Then the interval $[a,b]$ is essential. To see this note that no arc of intersection can be inessential on both $S$ and $T$ as otherwise we can reduce $|S \cap T|$. As $S$ is essential, there is no arc that is inessential in $T$ but essential in $S$. The interval $[u,v] \supset [a,b]$ certifies the fact that $[a,b]$ is a maximally essential interval.
\end{proof}

Since we have a maximally essential interval with respect to $\phi$, we can use the genus of $S$ to bound bridge number and distance of $T$. Recall that we always assume that $N = M(L)$ is irreducible and $\boundary$-irreducible.
\begin{theorem}\label{Thm: Ess Surf}
Suppose that $S \subset N$ is a $\sigma$-sloped essential surface, no component of which is a disc or sphere. Suppose that $T \subset N$ is a $\tau$-sloped Heegaard surface and that $\Delta=\Delta(\sigma,\tau) \geq 1$. Then:
\[
\Delta b_{\min}(T)(d_{\mc{AC}}(T) - 2) \leq \frac{4g(S) - 4|S| + 2|\boundary_0 S|}{|\boundary_L S|} + 2
\]
and
\[
\Delta b_{\min}(T)(d_{\mc{C}}(T) - 4) \leq \frac{8g(S) - 8|S| + 4|\boundary_0 S|}{|\boundary_L S|} + 4.
\]
\end{theorem}
\begin{proof}
By Lemma~\ref{Lem: T weak red}, we may assume that $T$ is strongly $\boundary$-irreducible. Let $\phi$ be a sweepout representing $T$. By Lemma~\ref{Lem: ess surface}, $\phi_S$ has a maximally essential interval. By Theorem~\ref{thm:distanceL} we have our conclusion.
\end{proof}

Specializing to the situation of knots in closed manifolds, we obtain:

\begin{theorem}\label{Thm: Ess Surface}
Suppose that $L$ is a knot with irreducible and $\boundary$-irreducible exterior in a closed, orientable 3-manifold $M$. Let $M'$ be the result of non-trivial Dehn surgery on $L$ with surgery distance $\Delta$. Then:
\begin{itemize}
   \item if there is no closed, connected, orientable, separating, essential surface of genus $g$ in the exterior of $L$, and if $M'$ contains such a surface, then
\[
\Delta b(L)(d_{\mc{AC}}(L) - 2) \leq \max(1,2g).
\]
and
\[
\Delta b(L)(d_\mc{C}(L) - 4) \leq \max(2, 4g)
\]
\item  if there is no closed, connected, orientable, non-separating, essential surface of genus $g$ in the exterior of $L$, and if $M'$ contains such a surface, then
\[
\Delta b(L)(d_{\mc{AC}}(L)-2)\leq max(1,4g-2)
\]
and
\[
\Delta b(L)(d_\mc{C}(L) - 4) \leq \max(2, 8g - 4)
\]
\end{itemize}
\end{theorem}
\begin{proof}

If $M'$ contains a closed, connected, orientable, separating essential surface of genus $g$, let $\ob{S}$ be such a surface, chosen so as to intersect $L^*$, the surgery dual to $L$, minimally out of all such surfaces. Otherwise, if $M'$ contains a closed, connected, orientable non-separating surface of genus $g$, let $\ob{S}$ be such a surface, chosen so as to intersect $L^*$ minimally, out of all such surfaces. We may assume that $\ob{S}$ intersects $L^*$. By the minimality property of $\ob{S}$, it is easy to see that $S = \ob{S} \cap M(L)$ is essential in $M(L) = N$.

Let $\ob{T}$ be any bridge surface for $(M,L)$ which minimizes the pair $(g(\ob{T}), b(\ob{T}))$ lexicographically. By Theorem \ref{Thm: Ess Surf}, we conclude
\[\begin{array}{rcl}
\Delta b(L)(d_\mc{AC}(T) - 2) &=& \Delta b(T)(d_\mc{AC}(T) - 2) \\
&\leq& \frac{4g - 4}{|\boundary_L S|} + 2\\
\end{array}
\]

If $g = 0$, then
\[
\Delta b(L)(d_\mc{AC}(T) - 2) < 2.
\]
Since the left hand side is either $-\infty$ or an integer, we have
\[
\Delta b(L)(d_\mc{AC}(T) - 2) \leq 1.
\]

If $g \geq 1$, then $\frac{4g - 4}{|\boundary_L S|} \leq 4g - 4$ and we have
\[
\Delta b(L)(d_\mc{AC}(T) - 2) \leq 4g - 4 + 2 = 4g - 2.
\]

If $S$ is separating, we can do better. For in that case, $|\boundary_L S| \geq 2$ and so, if $g \geq 1$, we have
\[
\Delta b(L)(d_\mc{AC}(T) - 2) \leq 2g - 2 + 2 = 2g.
\]

Thus, independent of the genus of $S$, if $S$ is separating we have:
\[
\Delta b(L)(d_\mc{AC}(T) - 2) \leq \max(1, 2g),
\]
and if $S$ is non-separating, we have
\[
\Delta b(L)(d_\mc{AC}(T) - 2) \leq \max(1, 4g-2).
\]
Since these inequalities hold for all bridge surfaces $\ob{T}$ realizing the Heegaard genus of $M$ and the bridge number of $L$, we have our desired inequalities for $d_\mc{AC}$. The inequalities for $d_\mc{C}$ follow from Lemma \ref{Lem: relating AC and C}.
\end{proof}

\section{The graphic}\label{sec:doublesweepout}

As usual, let $L \subset M$ be a link in a compact, orientable 3-manifold such that  $N=M(L)$ is irreducible and $\boundary$-irreducible. Let $S$ be either a sloped Heegaard surface for $N$ or a thick surface of a sloped generalized Heegaard surface for $N$. Let $\boundary_L N$ be the non-empty union of all components of $\boundary N \setminus \boundary M$ which have non-trivial intersection with $\boundary S$. (In this section, we do not require that $\boundary_L N = \boundary N \setminus \boundary M$.) Let $T$ be a sloped Heegaard surface for $N$ inducing the meridional (for $L \subset M$) multislope $\tau$ on $\boundary_L N$. Let $S$ have boundary slope $\sigma$ on $\boundary_L N$ and assume $\Delta(\sigma,\tau) > 0$.

Let $\phi_T$ be a sweepout of $N$ associated to $T$ and let $\phi_S$ be either a sweepout or a partial sweepout of $N$ associated to $S$. Assume that the boundary of each level surface of $\phi_S$ intersects the boundary of each level surface of $\phi_T$ minimally in each component $V$ of $\boundary_L N$. (This is possible, for example, by choosing a flat metric on each $V$ and requiring that the boundaries of the level sets be geodesics.)

Consider the product map $\phi_T \times \phi_S\co N \rightarrow [0,1]\times[0,1]$.   Each point $(t,s)$ in the interior of the square represents a pair of surfaces $T_t = \phi^{-1}_T(t)$ isotopic to the surface $T$ and $S_s = \phi^{-1}_S(s)$ isotopic to the surface $S$.  Recall that $(S_s)_\dn$ and $(S_s)_\up$ denote the intersection of compressionbodies to either side of $\ob{S}_s$ with $N$.

The \emph{graphic} is the union of $\boundary ([0,1]\times[0,1])$ with the subset of the square consisting of all points $(t,s)$ where $T_t$ and $S_s$ are tangent. A \emph{region} of the graphic is a component of the complement of the graphic in the unit square. We say that $\phi_T \times \phi_S$ is \emph{generic} if it is stable \cite{J09} on the complement of the spines and each line $\{t\}\times[0,1]$ and $[0,1]\times \{s\}$ contains at most one vertex of the graphic.  (Stable functions are those which, in the words of \cite{J09}, have the property that ``small perturbations do not change the topology''. Stable maps provide an alternative approach to the graphic from Cerf theory.) The vertices of the graphic in the interior of $[0,1]\times[0,1]$ are valence four (crossings) and valence two (cusps).  By general position of the spines, there are finitely many points of the graphic which do not have a neighborhood homeomorphic to $\mathbb{R}$. Moreover, we can consider the graphic to be a finite planar graph. The vertices in the corners of the boundary correspond to points of intersection in $\partial_L N$ between the union of components of $\boundary_0 N$ with a spine of $\phi_S$ and the union of components of $\boundary_0 N$ and a spine of $\phi_T$.  All other vertices in the boundary of the graphic are valence one or two. We model our analysis of the graphic on that of Johnson~\cite{J09}.

For $(t,s) \in (0,1) \times (0,1)$ a regular value of $\phi_T \times \phi_S$ we say that $T_t$ is \defn{essentially above} $S_s$ if there exists a component $l \subset S_s \cap T_t$ that is an essential arc or circle in $S_s$ but bounds a compressing or boundary compressing disc for $S_s$ contained in $(S_s)_{\dn}$. Similarly, we say that $T_t$ is \defn{essentially below} $S_s$ if there exists a component $l \subset S_s \cap T_t$ that is an essential arc or circle in $S_s$ and bounds a compressing or boundary compressing disc for $S_s$ contained in $(S_s)_{\up}$.

Let $Q_a$ ($Q_b$) denote the points in $(0,1) \times (0,1)$ such that $T_t$ is essentially above (below) $S_s$.

\begin{remark}\label{QatopQbbottom}
If the sloped Heegaard surface or thick surface $S$ is compressible into $S_\dn$, then any spine for $S_\dn$ will have an edge $e$. If for $t_0 \in (0,1)$ the surface $T_{t_0}$ is transverse to $e$ and if $\epsilon > 0$ is small, then a curve of $T_{t_0} \cap S_\epsilon$ corresponding to a meridian for $e$ will be an essential curve in $S_\epsilon$ bounding a disc in $(S_\epsilon)_\dn$. In this case, $(t_0, \epsilon) \in Q_a$. Thus, if $S$ is compressible in $S_\dn$, every open neighborhood of the bottom edge of the graphic has non-trivial intersection with $Q_a$. Similarly, if $S$ is compressible in $S_\up$, every open neighborhood of the top edge of the graphic has non-trivial intersection with $Q_b$. Thus, when $S$ is strongly $\boundary$-irreducible, every open neighborhood of the bottom edge of graphic has non-trivial intersection with $Q_a$ and every open neighborhood of the top edge of the graphic has non-trivial intersection with $Q_b$.
\end{remark}

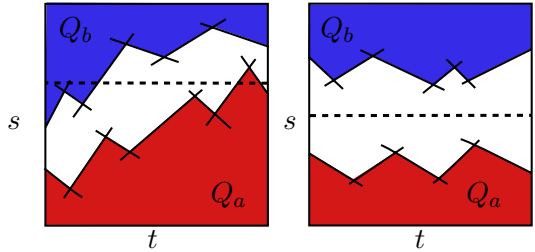
\begin{figure}[htb]
\labellist \small\hair 2pt
\pinlabel {$s$} [r] at 52 82
\pinlabel {$s$} [r] at 270 82
\pinlabel {$t$} [t] at 152 1
\pinlabel {$t$} [t] at 363 1
\pinlabel {$Q_b$} at 91 156
\pinlabel {$Q_b$} at 298 156
\pinlabel {$Q_a$} at 212 25
\pinlabel {$Q_a$} at 414 25
\endlabellist
\centering
\centering
  \includegraphics[width=3in]{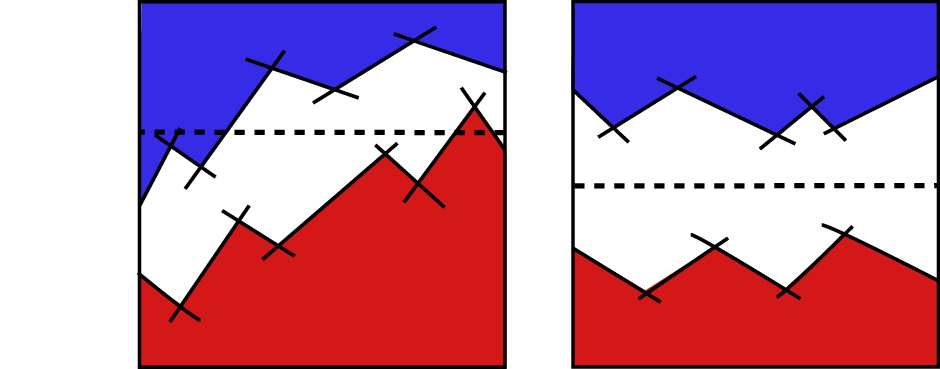}
  \caption{The graphic on the left represents the situation when $\phi_S$ essentially spans $\phi_T$. The graphic on the right represents the situation when $\phi_S$ essentially splits $\phi_T$. The dashed horizontal line on the left is at $s = s_0$ in the Definition \ref{def:spanning}. The dashed horizontal line on the right is at height $s = s_0$ in Definition \ref{def:splitting}. This figure is adapted from \cite[Figure 2]{J10}.}
  \label{fig:graphics}
\end{figure}

\subsection{Spanning}

\begin{definition}\label{def:spanning} The sweepout or partial sweepout $\phi_S$ \defn{essentially spans} the sweepout $\phi_T$ if there exist $t_1, t_2, s_0$ in $(0,1)$ such that $(t_1,s_0) \in Q_a$ and $(t_2,s_0) \in Q_b$, as on the left in Figure~\ref{fig:graphics}. Observe that if $ (Q_a \cap Q_b)$ is not empty, then $\phi_S$ essentially spans $\phi_T$ since $Q_a$ and $Q_b$ are open subsets of $(0,1)\times (0,1)$. Additionally, in the case that $ (Q_a \cap Q_b)$ is not empty, since $Q_a$ and $Q_b$ are open, we can assume $t_1 < t_2$.
\end{definition}

\begin{lemma} \label{lem:spanning} Suppose that $\phi_S$ is a sweepout or partial sweepout that essentially spans the sweepout $\phi_T$. Then $S$ is weakly $\boundary$-reducible.
\end{lemma}

\begin{proof}
Choose $s_0, t_1, t_2$ as in Definition~\ref{def:spanning}. Without loss of generality, $t_1 < t_2$. Let $\wihat{T} = T_{t_1} \cup T_{t_2}$. Let $l_k \subset S_{s_0} \cap T_{t_k}$ be a component such that $l_k$ is essential in $S_{s_0}$ and bounds a compressing or boundary compressing disc $D_k$ for $S_{s_0}$ lying in $(S_{s_0})_{\dn}$ if $k = 1$ and in $(S_{s_0})_{\up}$ if $k = 2$. Since $T_{t_1}$ and $T_{t_2}$ are disjoint, $\partial D_1$ and $\partial D_2$ are disjoint.
\end{proof}

\subsection{Splitting}\label{sec:splitting}

\begin{definition}\label{def:splitting}
The sweepout or partial sweepout $\phi_S$ \defn{essentially splits} the sweepout $\phi_T$ at $s_0 \in (0,1)$ if $Q_a$ and $Q_b$ are both non-empty and if  for all $t$, the point $(t, s_0)$ is in neither $Q_a \cup Q_b$ nor in the set of vertices of the graphic. The right side of Figure~\ref{fig:graphics} shows an example where $\phi_S$ essentially splits $\phi_T$ at some $s_0$.

The sweepout or partial sweepout $\phi_S$ \defn{weakly splits} the sweepout $\phi_T$ at $s_0 \in (0,1)$, if  the following hold:
\begin{itemize}
\item $Q_a$ and $Q_b$ are both non-empty,
\item for all $t \in (0,1)$, the point $(t, s_0) \not \in (Q_a \cup Q_b)$,

\item there exists $t_0 \in (0,1)$ such that $(t_0, s_0)$ is a valence 4 vertex of the graphic,
\item The point $(t_0, s_0)$ is in the intersection of the closures of $Q_a$ and $Q_b$.
\end{itemize}
\end{definition}

\begin{lemma}\label{lem:uvinterval}
Suppose $\phi_S$ is a generic sweepout or a partial sweepout that essentially splits or weakly splits the sweepout $\phi_T$ at $s_0$. If $\phi_S$ is a partial sweepout, assume that the thin surfaces for the associated sloped generalized Heegaard surface are incompressible and $\boundary$-incompressible. Assume also that both  $T$ and $S$ are strongly $\boundary$-irreducible. Then there exists a maximally essential interval $[a, b] \subset (0,1)$ for the sweepout $\phi_T \times \phi_S(t,s_0)$ of $N$.
\end{lemma}

\begin{proof}
For convenience, let $h_T (t)= \phi_T \times \phi_S(t,s_0)$. Note that, by construction for each regular value $t \in (0,1)$, $T_t = h_T^{-1}(t)$ is a surface in $N$ properly isotopic to $T$.

\textbf{Claim 1:} $h_T$ is adapted to $S_{s_0}$.

By the properties of the graphic and the definition of ``weakly splits'' and ``essentially splits'', $h_T$ satisfies conditions (1)-(4) of Definition~\ref{Defn: adapted}. We show that it also satisfies condition (5).

Suppose that $(t_0,s_0)$ is a valence four vertex of the graphic in the intersection of the closures of $Q_a$ and $Q_b$. Let $\epsilon > 0$ be small enough so that $(t_0, s_0)$ is the unique vertex of the graphic in the $\epsilon$ ball centered at $(t_0, s_0)$. Either $(t_0, s_0 + \epsilon) \in Q_b$ and $(t_0, s_0 - \epsilon) \in Q_a$ or vice versa. Since $S$ is strongly $\boundary$-irreducible, this implies that, in $S$, there is some component of $\mathbf{N}=T_{t_0} \cap S_{s_0 + \epsilon}$ which cannot be isotoped to be disjoint from some component of $\mathbf{S}=T_{t_0} \cap S_{s_0 - \epsilon}$.

To show that (5) holds, we must show that each component of $\mathbf{W} = T_{t_0 - \epsilon} \cap S_{s_0}$ can be isotoped in $T$ to be disjoint from each component of $\mathbf{E} = T_{t_0 + \epsilon} \cap S_{s_0}$. (The following argument is essentially that of \cite[Lemma 5.6]{RS}.) To see this, note that to  go from $\mathbf{S}$ to $\mathbf{N}$ we attach two disjoint bands $b_1$ and $b_2$. To go from $\mathbf{S}$ to $\mathbf{E}$ we attach $b_1$ and to go from $\mathbf{S}$ to $\mathbf{W}$ we attach $b_2$. To go from $\mathbf{E}$ to $\mathbf{N}$ we attach $b_2$ and to go from $\mathbf{W}$ to $\mathbf{N}$ we attach $b_1$. If the ends of $b_1$ and $b_2$ are adjacent to different components of $\mathbf{S}$ or if the ends of $b_1$ and $b_2$ are adjacent to the same side of the same component, then by moving those components slightly in the direction opposite the side to which the bands are incident, we see that $\mathbf{S}$ is disjoint from $\mathbf{N}$ in $S$. But this contradicts the observation that some component of $\mathbf{S}$ cannot be isotoped to be disjoint from some component of $\mathbf{N}$.  Then, an end of $b_1$ and an end of $b_2$ are adjacent to opposite sides of the same component $\nu$ of $\mathbf{S}$. If the other ends of $b_1$ and $b_2$ are adjacent to a component $\nu'$ of $\mathbf{S}$ then those ends must be adjacent to opposite sides of $\nu_2$, for else one of $b_1$ or $b_2$ would be a band with ends attached to opposite sides of a component of $\mathbf{E}$ or $\mathbf{W}$, a contradiction to the orientability of $S$ and $T$. Thus, the arcs and loops that result from attaching $b_1$ to $\mathbf{S}$ can be isotoped to be disjoint from the arcs and loops that result from attaching $b_2$ to $\mathbf{S}$ (push each component a little in the direction that the bands approach from). Hence, each component of $\mathbf{E}$ is isotopically disjoint from each component of $\mathbf{W}$. See Figure \ref{fig:doubleBand} for an example.\qed(Claim 1)

\begin{figure}[hbt]
\begin{center}
\includegraphics[scale=.3]{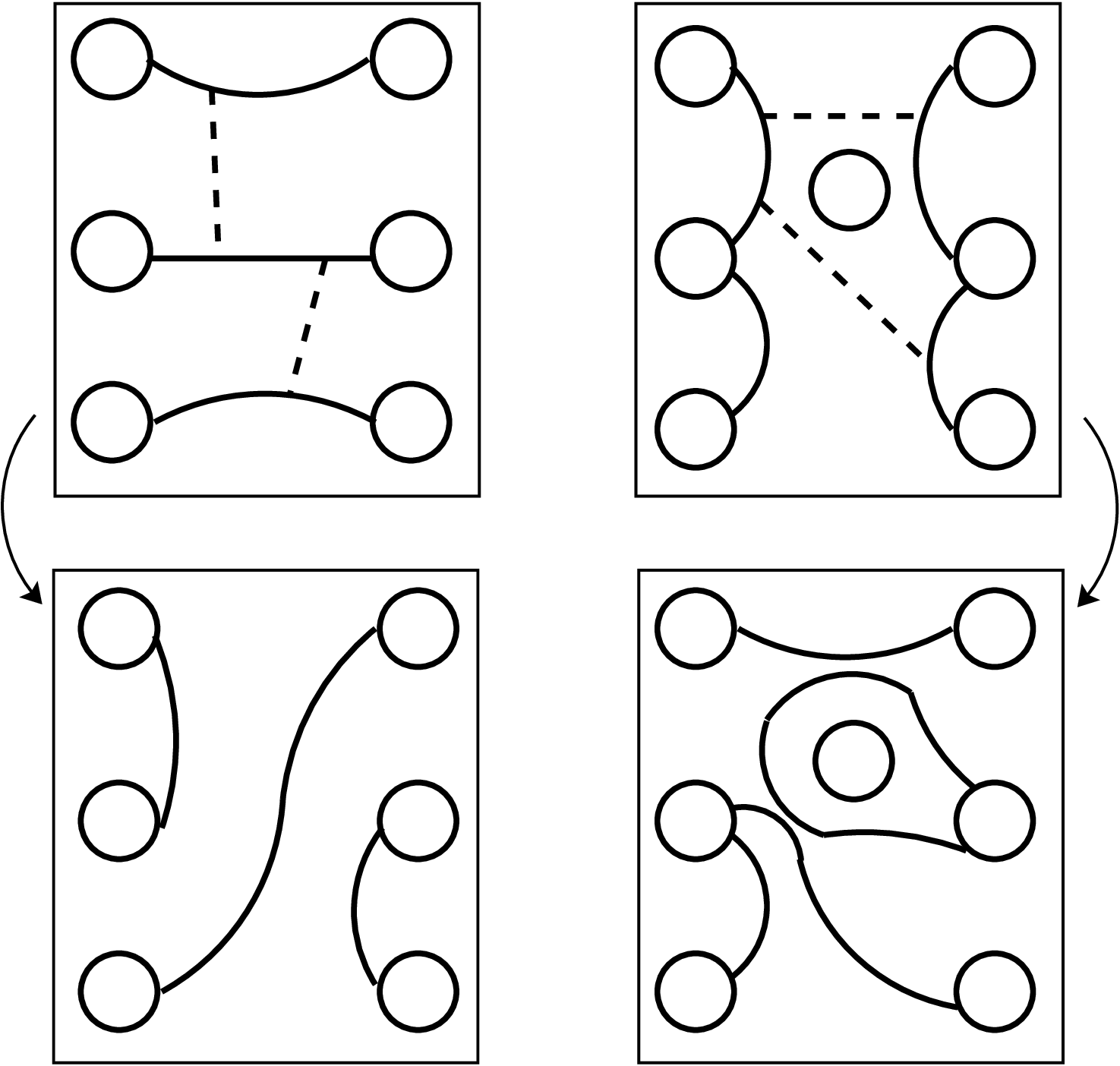}
\caption{An example showing how attaching bands creates disjoint curves depending on the positioning of the endpoints of the bands. The top diagram on the left and right shows pre-active arcs with the bands. The second row on each side shows the post-active arcs. On the left, the bands are incident to opposite sides of the same pre-active arc and on the right the bands are incident to the same side of a pre-active arc.}
\label{fig:doubleBand}
\end{center}
\end{figure}

\textbf{Claim 2:} Every arc or loop in the intersection $S_{s_0} \cap T_t$ that is trivial in $T_t$ must also be trivial in $S_{s_0}$.

Let $F_0$ be the (possibly empty) thin surface immediately below $S$ and let $F_1$ be the (possibly empty) thin surface immediately above $S$. Recall that they are incompressible and $\boundary$-incompressible in $N$.

If there were a trivial loop or arc of $S_{s_0} \cap T_t$ in $T_t$ that was essential in $S_{s_0}$, then an innermost such loop or outermost such arc $\beta$ would bound or cobound a disc $D$ in $T_t$. Isotope $D$ fixing $\beta$ so that the number of components of  $D\cap (F_0 \cup F_1 \cup S_{s_0})$ is minimized. Suppose $\alpha$ is an outermost arc or innermost loop of $D\cap (F_0 \cup F_1 \cup S_{s_0})$ in $D$. If $\alpha$ is contained in $F_0\cup F_1$, then, by irreducibility of $N$ and incompressibility and $\boundary$-incompressibility of $F_0 \cup F_1$, we can isotope $D$ to decrease the number of components of  $D\cap (F_0 \cup F_1 \cup S_{s_0})$. If $\alpha$ is contained in $S_{s_0}$, then it must be trivial in $S_{s_0}$ since $\beta$ is an innermost essential loop or an outermost essential arc. Since $\alpha$ is trivial in both $S_{s_0}$ and $D$ and since $N$ is irreducible and boundary irreducible, there is an isotopy of $D$ that decreases the number of components of $D\cap (F_0 \cup F_1 \cup S_{s_0})$. In either case, we arrive at a contradiction, so we can assume the interior of $D$ is disjoint from $F_0 \cup F_1 \cup S_{s_0}$. Hence, after an isotopy, the disc $D$ is completely contained in either $(S_{s_0})_{\dn}$ or $(S_{s_0})_{\up}$, contradicting the assumption that $(t, {s_0})$ is disjoint from $Q_a$ and $Q_b$ for all $t$. We conclude that every arc or loop of $S_{s_0} \cap T_t$ is either essential in $T_t$ or trivial in both surfaces.\qed(Claim 2)

On the other hand, a loop or arc in the intersection may be trivial in $S_{s_0}$ but essential in $T_t$. In fact, for values of $t$ near $0$, such loops/arc must exist and will bound compressing discs and boundary compressing discs, respectively, in $(T_t)_\dn$. For $t$ near $1$, such loops and arcs will bound discs in $(T_t)_\up$. Since $T$ is strongly $\boundary$-irreducible, Gabai's argument (as in the proof of Lemma~\ref{Lem: ess surface}) shows that there exists a maximally essential interval.\end{proof}

Finally, we observe:
\begin{corollary}\label{splitbound}
Let $T$ be a $\tau$-sloped Heegaard surface for $N$ and let $S$ be a strongly $\boundary$-irreducible thick surface for a generalized $\sigma$-sloped Heegaard splitting of $N$ satisfying conclusions (1), (2), (4), (5) of Theorem \ref{ThinPosition1}. Assume that $\Delta(\sigma,\tau) \geq 1$. Then
\[
\Delta b_{\min}(T)(d_{\mc{AC}}(T) - 2) \leq \frac{4g(S) - 4|S| + 2|\boundary_0 S|}{|\boundary_L S|} + 2
\]
and
\[
\Delta b_{\min}(T)(d_\mc{C}(T) - 4) \leq \frac{8g(S) - 8|S| + 4|\boundary_0 S|}{|\boundary_L S|} + 4
\]
\end{corollary}

\begin{proof}
If $T$ is weakly $\boundary$-reducible, by Lemma~\ref{Lem: T weak red}, the desired inequalities hold. We assume, therefore, that $T$ is strongly $\boundary$-irreducible.

Let $\phi_S$ and $\phi_T$ be (partial) sweepouts corresponding to $S$ and $T$ so that $\phi_T \times \phi_S$ is generic. By hypothesis, if $S$ is a thick surface for a sloped generalized Heegaard surface, then the thin surfaces for that sloped generalized Heegaard surface are essential. Since $L$ is a link, they are incompressible and $\boundary$-incompressible. Since $S$ is strongly $\boundary$-irreducible, by Lemma~\ref{lem:spanning}, $\phi_S$ does not span $\phi_T$ and, by definition of strong $\boundary$-irreducibility, $Q_a$ and $Q_b$ are both non-empty.

Suppose that $A$ and $B$ are two distinct regions of the graphic adjacent to a vertex $(t_0, s_0)$ of the graphic of valence 2. As $(t_0,s)$ passes through $(t_0, s_0)$, a center tangency and a saddle tangency between the surfaces $T_{t_0}$ and $S_{s}$ are cancelled. Hence, a point of $A$ belongs to $Q_a$ or $Q_b$ if and only if a point of $B$ belongs to $Q_a$ or $Q_b$ respectively.

Let $[0,s_0]$ be the largest interval such that $[0,1]\times [0,s_0]$ is disjoint from $Q_b$. Since $\phi_S$ does not span $\phi_T$, $s_0 \neq 0$. Note that the line defined by $s = s_0$ must pass through a vertex of the graphic. By the previous paragraph, that vertex has valence 4. Let $\epsilon > 0$ be small enough so that there is no vertex of the graphic in the region $\{(t,s) : s_0 - \epsilon \leq s < s_0\}$.  By the choice of $s_0$, the line $s =  s_0 - \epsilon$ is disjoint from $Q_b$. If it is also disjoint from $Q_a$, then $\phi_S$ essentially splits $\phi_T$ at $s_0 - \epsilon$. If the line $s = s_0 - \epsilon$ intersects $Q_a$, then $\phi_S$ weakly splits $\phi_T$ at $s_0$. In either case, by Lemma~\ref{lem:uvinterval},  $\phi_T$ is adapted to $S_{s_0}$ and there exists $s'_0 \in (0,1)$ such that there exists a maximally essential interval $[a, b] \subset (0,1)$ for the sweepout $\phi_T \times \phi_S(s'_0)$ of $N$. Our result then follows from Theorem~\ref{thm:distanceL}.
\end{proof}

\section{Alternately sloped Heegaard surfaces}\label{sec: bridge surface bounds}

Rather than examining  elementary consequences of Corollary~\ref{splitbound}, we move on to consider the general case for when $\ob{S}$ is a, possibly weakly reducible, bridge surface. For simplicity, we examine only the case when $L$ is a knot in a closed 3-manifold.

\begin{theorem}\label{Thm:Nonmerid bridge}
Assume that $L$ is a knot in a closed, orientable 3-manifold $M$ having irreducible and $\boundary$-irreducible exterior. Let $M'$ be the result of a non-trivial Dehn surgery on $L$ and let $g$ be the Heegaard genus of $M'$. Let $L^*$ be the surgery dual to $L$. Then one of the following holds:
\begin{enumerate}
\item We have
\[
b(L)(d_{\mc{AC}}(L) - 2) \leq \max(1, 2g)
\]
and
\[
b(L)(d_{\mc{C}}(L) - 4) \leq \max(2, 4g)
\]
\item $M(L)$ has a Heegaard surface of genus $g$.
\item $M$ contains an essential surface of genus strictly less than $g$ intersecting $L$ transversally at most twice.
\item $M(L)$ contains an essential surface of genus strictly less than $g$ and $L^*$ is isotopic into the surface. Furthermore, the dual surgery on $L^*$ creating $M$ from $M'$ has surgery slope equal to the slope of the boundary of the surface after $L^*$ has been isotoped to lie in it.
\end{enumerate}
\end{theorem}
\begin{proof}

As before we let $N=M(L)$ and for any surface $\ob{X}$ in $M$ we let $X=N \cap \ob{X}$. In what follows, whenever we have shown an  inequality for $d_\mc{AC}$, we will deduce the corresponding inequality for $d_\mc{C}$ by using Lemma \ref{Lem: relating AC and C}.

Let $L^* \subset M'$ be the surgery dual to $L$. Let $\ob{S}$ be a bridge surface for $(M',L^*)$ chosen so that $\ob{S}$ has genus $g$ and, out of all such bridge surfaces, minimizes $|L^* \cap \ob{S}|$. In particular, $\ob{S}$ is not stabilized or perturbed. If $\ob{S}$ is removable, $\ob{S}$ is isotopic to a Heegaard surface for $N$ and we have conclusion (2). Assume, therefore, that $\ob{S}$ is not removable.  Let $\sigma$ be the slope in $\boundary_L N$ which is represented by $\boundary S$.

Let $\ob{T}$ be a bridge surface for $(M,L)$ so that $g(\ob{T})$ is the Heegaard genus of $M$ and so that $b(\ob{T}) = b(L)$.  We will show that either conclusions (2), (3), or (4) hold or that:
\begin{equation}\label{Ineq-AC}
b(T)(d_{\mc{AC}}(T) - 2) \leq \max(1, 2g)
\end{equation}

In which case, we can conclude that if (2), (3) and (4) do not hold, then since this inequality holds for all bridge surfaces $\ob{T}$ such that $g(\ob{T})$ is the genus of $M$ and $b(\ob{T}) = b(L)$, we will have deduced that conclusion (1) holds.

By Theorem~\ref{ThinPosition1} (applied to $\ob{S}$) there is a $\sigma$-sloped generalized Heegaard splitting $(\mc{F},\mc{S})$ with the following properties:
\begin{enumerate}
\item[(T1)] Each component of $\mc{F}$ is essential in $N$.
\item[(T2)] Each component of $\mc{S}$ is strongly irreducible in its component of the closure of $N \setminus \inter{\eta} (\mc{F})$.
\item[(T3)] Each component $J$ of $\mc{S} \cup \mc{F}$ has genus no greater than $g$, $|\boundary_L J| \leq |\boundary_L S|$ and $J$ is obtained by compressing $S$ in $N=M(L)$. Furthermore if $\mc{F}$ is non-empty, at least one such compression has taken place for each component of $\mc{F} \cup \mc{S}$.
\item[(T4)] No component of $\mc{S}$ is perturbed or removable in its component of the closure of $N \setminus \inter{\eta} (\mc{F})$.
\item[(T5)] Each component of $\mc{S}$ is separating in $N$.
\end{enumerate}

We first consider the case when there is a thin surface with boundary. In this case we use Theorem~\ref{Thm: Ess Surf} to get the desired bound. Subsequently, we may then assume that $\mc{F}$ is closed. We then single out a particular component $S^*$ of $\mc{S}$. The remaining cases correspond to whether $S^*$ is strongly $\boundary$-irreducible or cancellable.

\textbf{Case 1:} There is a component of $\mc{F}$ with non-empty boundary.

Let $\ob{C} \subset M' \setminus \inter{\eta}(\ob{\mc{F}} \cup \ob{\mc{S}})$ be a compressionbody component so that $\boundary_- \ob{C} \cap L^* \neq \nil$. Since $\boundary_+ \ob{C}$ is separating in $M'$, $|\boundary_- \ob{C} \cap L^*| \geq 2$. Let $\ob{F} = \boundary_- \ob{C}$. ($\ob{F}$ may be disconnected). Observe that no component of $F$ is a disc or sphere, since $F$ is essential in $N$ and $N$ is irreducible and boundary-irreducible. Since $\ob{F}$ is obtained by compressing $\boundary_+ \ob{C}$, we have
\[ g(\ob{F}) \leq g(\boundary_+ \ob{C}) \leq g(\ob{S}) = g.\]

Apply Theorem~\ref{Thm: Ess Surf} to $F$ in place of $S$ to deduce
\[
b(T)(d_{\mc{AC}}(T) - 2) \leq \frac{4g(F) - 4|F|}{|\boundary_L F|} + 2
\]
If $g(F) = 0$, we have
\[
b(T)(d_{\mc{AC}}(T) - 2) < 2
\]
Since the left hand side is an integer, we conclude that \[b(T)(d_\mc{AC}(T) - 2) \leq 1,\] as desired.

If $g(F) \geq 1$, since $|\boundary_L F| \geq 2$, we have
\[
b(T)(d_\mc{AC}(T) - 2) \leq 2g(F) \leq 2g(S),
\]
as desired. \qed(Case 1)

Henceforth, we assume that $\mc{F}$ is closed. The 3-manifold $N \setminus\mc{F}$ has a component $N'$ which intersects $L^*$. Let $S^* = \mc{S} \cap N'$.

\textbf{Case 2:} $S^*$ is weakly $\boundary$-reducible in $N$.

Since $\mc{\ob{F}}$ is disjoint from $L^*$ and $L^*$ is connected, then $L^*$ is contained in $N'$. Since the boundary of the closure of $N'$ is essential in $N$, $L^*$ is contained in $N'$ and $N$ is irreducible, then any compressing or boundary compressing disk for $S^*$ in $N$ can be isotoped to lie in $N'$. Thus, $S^*$ is weakly $\boundary$-reducible but strongly irreducible in $N'$. Hence, $S^*$ must be cancellable or perturbed in $N'$. By property (T4), $S^*$ cannot be perturbed, so it must be cancellable. Let $\ob{R}$ be a cancelled bridge surface  (obtained by cancelling $\ob{S^*}$) and let $\rho$ be the slope of $\boundary R$ on $\boundary_L N$. By Lemma~\ref{Lem:Cancellable}, the surface $R$ is essential in $N$ and one of the following holds:
\begin{enumerate}
\item[(Ca)] after Dehn filling $\boundary M(L)$ with slope $\rho$ and capping $\boundary R$ with discs, we obtain an incompressible surface in the filled manifold.
\item[(Cb)] $L^*$ is isotopic into a component $\ob{F} = F$ of $\ob{\mc{F}}$ adjacent to $\ob{S^*}$. Furthermore, in $M'$, $\ob{F}$ and $\ob{S^*}$ bound a product compressionbody.
\end{enumerate}

\textbf{Case 2a:} $\Delta(\rho,\tau) \geq 1$.

Since $\Delta(\rho,\tau) \geq 1$, we can apply Theorem~\ref{Thm: Ess Surf} to $R$. In this case, (following the arithmetic of Case 1), we conclude that
\[
b(T)(d_{\mc{AC}}(T) - 2) \leq \max(1,2g(R)) \leq \max(1,2g(S))
\]
since $|\boundary R| \geq 2$ and $g(R) = g(S^*) \leq g(S)$.

\textbf{Case 2b:} $\Delta(\rho,\tau) = 0$.

If $R$ is connected, then it has genus one less than the genus of $S^*$. If $R$ is disconnected, then the sum of the genera of its components is equal to the genus of $S^*$.  Since $L^*$ is connected and since its exterior is $\boundary$-irreducible, no component of $R$ is a disc.

If (Ca) occurs, we have Conclusion (3). Assume, therefore, that (Cb) occurs.

The surface $S^*$ is obtained by compressing $S$ in $N$. Since $\ob{\mc{F}} \neq \nil$, there is at least one compression. Let $D \subset N$ be the first compressing disc for $S$ in a sequence of compressions leading to $S^*$ and let $S'$ be the surface resulting from compressing along $D$. If $\boundary D$ doesn't separate $S$, then the compression strictly reduces the genus of $S$ and we can conclude that $g(S^*) \leq g(S') < g(S)$. If $\boundary D$ does separate $S$, then either both components of $S'$ have genus strictly less than the genus of $S$ or one of them is a planar surface $P$. Since $|\ob{S^*} \cap L^*| = 2$ and since $\ob{S^*}$ is obtained via compressions in $N$ from $\ob{S}$, we must have $|\ob{S} \cap L^*| = 2$. Each thick surface in $\ob{\mc{S}}$ is separating, so $P$ is an annulus and $S' \setminus P$ is closed. Continuing on through the sequence of compressions, we discover that if $g(S^*) = g(S)$ then $S^*$ must be a closed surface. But this contradicts the fact that $|L^* \cap \ob{S^*}| = 2$. Hence, $g(\ob{R}) = g(\ob{S^*}) < g(S)$. Since $\ob{F}$ is parallel to $\ob{R}$, we have $g(\ob{F}) < g(S)$. Since all components of $\mc{F}$ are essential in $N$ and are closed surfaces, we have Conclusion (4).

\textbf{Case 3:} $S^*$ is strongly $\boundary$-irreducible in $N$.

By Corollary~\ref{splitbound}, we have
\[
b(T)(d_{\mc{AC}}(T) - 2) \leq \frac{4g(S^*) - 4}{|\boundary_L S^*|} + 2
\]

Since $|\boundary_L S^*| \geq 2$ and $g(S^*) \leq g(S)$, we again get Inequalities \eqref{Ineq-AC}. \qed(Case 3)
\end{proof}

\section{Bounding distance and bridge number}\label{sec:ImprovingBridgeBound}
In Theorem~\ref{Thm:Nonmerid bridge}, the possibility that there is an essential meridional surface of genus $g$ with two boundary components, that there is a Heegaard surface of genus $g$ for the complement of $L$, or that there is a closed essential surface in the exterior of $L$ ruin the ability to bound the distance of $L$ purely in terms of the genus $g(S)$. In this section, we observe that a theorem of Tomova allows us to bound distance even in these cases.

\begin{theorem}[Theorems 5.7 and 10.3 of~\cite{To07}]
\label{MaggyThm}
Suppose that $L$ is a non-trivial knot in a closed, orientable, irreducible 3-manifold $M$. Let $\ob{T}$ be a bridge surface for $(M,L)$.
\begin{itemize}
\item Suppose that $\ob{S} \subset M$ is a surface transverse to $L$ such that $S \subset M(L)$ is not a sphere and is essential in $M(L)$. Then
\[
d_\mc{C}(\ob{T}) \leq \max(3, 2g(S) + |\ob{S} \cap L|).\]
 In fact, if, in addition, $\ob{S} = S$ and if $S$ is a torus, then $d_\mc{C}(\ob{T}) \leq 2$.
\item If $\ob{S}$ is a  Heegaard surface for $M(L)$, then
\[
d_{\mathcal{C}}(\ob{T}) \leq 2g(\ob{S}).
\]
\end{itemize}
\end{theorem}

We can combine these results with our previous work to obtain a bound on distance.

\begin{theorem}\label{Bounding dist}
Assume that $M$ is closed, orientable and irreducible and that $L \subset M$ is a knot with irreducible and $\boundary$-irreducible exterior $N = M(L)$. Let $M'$ be the result of non-trivial Dehn surgery on $L$ and let $g$ be the Heegaard genus of $M'$.  Then
\[ d_{\mc{C}}(L) \leq \max\Big(\frac{2}{b(L)} + 4, \frac{4g}{b(L)} + 4, 2g+ 2\Big).\]
\end{theorem}
\begin{proof}
Since $3$ is always less than the right hand side of the desired inequality, we may assume that $d_\mc{C}(L) \geq 4$.  By Theorem~\ref{Thm:Nonmerid bridge}, one of the following occurs:
\begin{enumerate}
\item We have
\[
b(L)(d_{\mc{C}}(L) - 4) \leq \max(2, 4g)
\]
\item $M(L)$ has a Heegaard surface of genus $g$.
\item $M$ contains an essential surface of genus strictly less than $g$ intersecting $L$ transversally at most twice.
\item $M(L)$ contains an essential surface of genus at most $g-1$ and $L^*$ is isotopic into the surface. Furthermore, the dual surgery on $L^*$ creating $M$ from $M'$ has surgery slope equal to the slope of the boundary of the surface after $L^*$ has been isotoped to lie in it.
\end{enumerate}

If conclusion (1) occurs, then
\[ d_{\mc{C}}(L) \leq \frac{\max(2, 4g)}{b(L)} + 4, \]
as desired.

Suppose that conclusion (2) occurs. Let $\ob{T}$ be a minimal bridge surface for $(M,L)$ and let $S = \ob{S}$ be a Heegaard surface of genus $g$ for $M(L)$.  By Theorem~\ref{MaggyThm},
\[ d_\mc{C}(\ob{T}) \leq 2g \leq \max\Big(\frac{2}{b(L)} + 4, \frac{4g}{b(L)} + 4, 2g+ 2\Big). \]

Suppose now that conclusion (3) occurs. Among all surfaces of genus at most $g(S) - 1$ which intersect $L$ at most twice and which are essential in $M$, choose $\ob{F}$ to minimize the pair $(g(\ob{F}), |\ob{F} \cap L|)$ lexicographically.

We begin by showing that $F = \ob{F} \cap N$ is essential.

Suppose that $D$ is a compressing disc for $F$. Since $\ob{F}$ is essential in $M$, $\boundary D$ bounds a disc in $\ob{F}$.  Thus $\boundary D$, together with boundary components of $F$ bounds a subsurface $A\subset F$ which is either an annulus or a pair of pants.  If $A$ is an annulus, then $N$ is $\boundary$-reducible, a contradiction. If $A$ is a pair of pants, compressing $F$ using $D$ creates a meridional annulus $B$ in $N$. If $B$ were not $\partial$-parallel, then $\ob{B}$ would be a surface of lower complexity, contradicting our choice of $\ob{F}$. Hence, $B$ is $\boundary$-parallel. There is, therefore, an isotopy of $\ob{F}$ to be disjoint from $L$, contradicting our choice of $\ob{F}$. Hence, $F$ is incompressible.

If $F$ is $\boundary$-compressible then, since $\boundary N$ is a torus, either $F$ is compressible or $F$ is a $\boundary$-parallel annulus. The former possibility contradicts the previous paragraph and the latter means that we can isotope $\ob{F}$ to be disjoint from $L$. That also contradicts our choice of $\ob{F}$. Hence, $F$ is $\boundary$-incompressible.

Consequently, $F$ is essential.

By Theorem~\ref{MaggyThm}, if $\ob{T}$ is a bridge surface for $(M,L)$, then
\[d_{\mc{C}}(\ob{T}) \leq \max(3, 2g(\ob{F}) + 2) \leq \max \Big(\frac{2}{b(L)} + 4, \frac{4g}{b(L)} + 4, 2g+ 2\Big).\]

Since, whenever $\ob{T}$ is a minimal bridge surface for $(M,L)$ we have
\[
d_\mc{C}(\ob{T}) \leq \max\Big(\frac{2}{b(L)} + 4, \frac{4g}{b(L)} + 4, 2g+ 2\Big),
\]
we have
\[
d_\mc{C}(L) \leq \max\Big(\frac{2}{b(L)} + 4, \frac{4g}{b(L)} + 4, 2g+ 2\Big),
\]
as desired.

Finally, if conclusion (4) occurs we deal with it in a similar way to how we dealt with conclusion (2), only obtaining a better bound.
\end{proof}

We can now bound the distance of knots admitting exceptional and cosmetic surgeries.

\begin{theorem}\label{Bounding distance - special}
Suppose that $L$ is a knot in a closed, orientable, irreducible 3-manifold $M$ which has irreducible and $\boundary$-irreducible exterior. Let $M'$ be a 3-manifold obtained by non-trivial Dehn surgery on $L$. Then the following hold:
\begin{enumerate}
\item If $M = S^3$ and if $M'$ is a lens space, then $d_\mc{C}(L) \leq 4$.
\item If $M = S^3$ and if $M'$ is a small Seifert fibered space, then $d_\mc{C}(L) \leq 6$.
\item If $M$ is hyperbolic and if $M'$ is not hyperbolic, then $d_\mc{C}(L) \leq 12$.
\item If the Heegaard genus of $M'$ is at most $g$, then $d_{\mc{C}}(L) \leq \max(6,4g + 4)$.
\end{enumerate}
\end{theorem}
\begin{proof}
By the Geometrization Theorem of Thurston and Perelman ~\cites{P1,P2,P3} if $M'$ is non-hyperbolic, then it is reducible, toroidal, or a small Seifert-fibered space. Small Seifert fibered spaces have Heegaard genus at most 2 \cite{BZ}. If $M'$ is reducible or toroidal, conclusion (3) follows from Theorem~\ref{Thm: Ess Surface}.

By Theorem \ref{MaggyThm}, if $M(L)$ contains an essential meridional annulus or an essential torus, then $d_\mc{C}(L) \leq 4$. Assume, therefore, that there is no such annulus or torus. Also, if $M = S^3$, we may assume that $b(L) \geq 3$, as otherwise $d_\mc{C}(L) = -\infty$.

Suppose, first, that $M'$ is a lens space. Since $M'$ is a lens space, it has a genus $g = 1$ Heegaard surface $\ob{S}$. Since $M(L)$ is irreducible and $\boundary$-irreducible, it does not have a Heegaard surface of genus $g=1$ or an essential sphere. If $M$ contains an essential sphere intersecting $L$ once or twice, then we either violate the irreducible and $\boundary$-irreducibility of $M(L)$ or $M(L)$ contains an essential annulus, a contradiction.  By Theorem \ref{Thm:Nonmerid bridge},
\[
b(L)(d_\mc{AC}(L) - 2) \leq 2.
\]
Hence,
\[
d_\mc{AC}(L) \leq \frac{2}{b(L)} + 2 \leq 4.
\]
We obtain conclusion (3) by applying Lemma \ref{Lem: relating AC and C} and concluding that $d_\mc{C}(L) \leq 8$. If, in addition, $M = S^3$, then $b(L) \geq 3$, and since $d_\mc{AC}(L)$ is an integer, we have $d_\mc{AC}(L) \leq 2$. By Lemma \ref{Lem: relating AC and C}, $d_\mc{C}(L) \leq 4$, giving us conclusion (1).

Similarly, if $M'$ is a small Seifert fibered space, conclusions (2) and (3) follow from Theorem~\ref{Bounding dist} by setting $g(S) = 2$, and using the fact that if $M = S^3$ then $b(L) \geq 3$ and that $d_\mc{C}(L)$ is an integer or $-\infty$. More generally, we note that if $M'$ has Heegaard genus $g$, then
\[
\max\left(\frac{2}{b(L)} + 4, \frac{4g}{b(L)} + 4, 2g+ 2\right) \leq \max(6,4g + 4),
\]
and so conclusion (4) also follows from Theorem ~\ref{Bounding dist}.
\end{proof}

One particular application of Conclusion (4) from Theorem \ref{Bounding distance - special} is the following Corollary. The proof follows immediately from the fact that the Heegaard genus of $S^3$ is 0. The corollary says, in essence, that knots of high bridge distance do not admit non-trivial $S^3$ surgeries.

\begin{corollary}\label{Cor: S^3 surgeries}
Suppose that $L$ is a knot in a closed, orientable, irreducible 3-manifold $M$ which has irreducible and $\boundary$-irreducible exterior. If a non-trivial Dehn surgery on $L$ produces $S^3$, then $d_\mc{C}(L) \leq 6$.
\end{corollary}

We can also obtain an improved bound on the bridge number of knots admitting non-trivial reducible or toriodal surgeries.
\begin{theorem}\label{apps1}
Suppose that $L \subset S^3$ is a knot with $d_\mc{C}(L) \geq 3$. Then:
\begin{enumerate}
\item If non-trivial Dehn surgery on $L$ produces a reducible 3-manifold, then $b(L) \leq 5$.
\item If non-trivial Dehn surgery on $L$ produces a toroidal 3-manifold, then $b(L) \leq 6$. Furthermore, if the surgery slope is non-longitudinal, then $b(L) \leq 5$.
\end{enumerate}
\end{theorem}

\begin{proof}
Since $d_\mc{C}(L) \geq 3$, by definition $L$ is not the unknot, and so has irreducible and $\boundary$-irreducible exterior. Isotope $L$ into a bridge position with respect to a Heegaard sphere $\ob{T}$ in $S^3$ minimizing $b(L)$ and with $d_\mc{C}(T) \geq 3$. If non-trivial Dehn surgery on $L$ produces a reducible 3-manifold, in the surgered manifold choose a reducing sphere intersecting the core of the surgery torus minimally. Since the exteriors of knots in $S^3$ are irreducible, the intersection $S$ of the reducing sphere with $S^3(L)$ is an essential planar surface in $S^3(L)$. Since $\ob{T}$ has $d_\mc{C}(L) \geq 3$, the surface $T$ is strongly $\boundary$-irreducible.  By Lemma~\ref{Lem: ess surface}, there is a sweepout by $T$ adapted to $S$  so that there is a maximally essential interval.  By Theorem~\ref{Bounding Bridge},
\[
b(L) - 4 \leq \Delta(b(L) - 4) \leq \frac{4g(S) - 4}{|\boundary_L S|} + 2 < 2.
\]
Since $b(L)$ is an integer, $b(L) \leq 5$.

If $L$ has a toroidal surgery, then $M'$ has an essential torus $\ob{S}$ such that $S$ is essential in $N$. If $\boundary_L S = \nil$, then, by Theorem \ref{MaggyThm}, $d_\mc{C}(L) \leq 2$, a contradiction. Hence, we can assume $\boundary_L S \neq \nil$. The proof is now nearly identical to the proof of the reducing case in the previous paragraph.
\end{proof}

Likewise, we can also obtain a bound on the bridge number of high distance knots in terms of the genus of the knot.

\begin{theorem}\label{Bridge number and genus}
Suppose that $L$ is a knot in a homology sphere $M$ with irreducible and $\boundary$-irreducible exterior. If $d_\mc{C}(L) \geq 3$ then
\[
b(L) \leq 4g(L)+2,
\]
where $g(L)$ is the Seifert genus of $L$.
\end{theorem}

The proof is an easy combination of Theorem~\ref{Bounding Bridge} and Lemma~\ref{Lem: ess surface}, as in the proof of Theorem~\ref{apps1}.

We conclude with a few consequences for exceptional surgeries and cosmetic surgeries.

\begin{corollary}\label{Cor: exceptional bridge}
Suppose that $L \subset M$ is a hyperbolic knot in a closed non-Haken 3-manifold such that $d_\mc{C}(L) \geq 3$. Let $M'$ be a 3-manifold obtained by non-trivial Dehn surgery on $M$. Then the following hold:
\begin{enumerate}
\item If $M'$ is non-hyperbolic and if the tunnel number of $L$ is at least 2, then $b(L) \leq 8$
\item If $M'$ has Heegaard genus at most the Heegaard genus $g$ of $M$ and if $M(L)$ does not contain an essential surface of genus at most $g-1$, then $b(L) \leq 2g + 4$.
\end{enumerate}
\end{corollary}
\begin{proof}
Since $M$ is non-Haken, there is no essential surface in $M$ intersecting $L$ at most twice transversally. Since $L$ is hyperbolic, there is no essential sphere, annulus, disc, or torus in $M(L)$.  If $M'$ is non-hyperbolic it contains an essential sphere, an essential torus, or it has Heegaard genus at most 2. Under the hypotheses of (1), the tunnel number of $L$ is at least 2, which implies that there is no Heegaard surface for the exterior of $L$ having genus 2.

If the tunnel number of $L$ is at most $g-1$, then there is a minimal genus Heegaard surface for $M$ which is also a Heegaard surface for $M(L)$. The knot $L$ can be isotoped to be in bridge position with respect to this surface so $b(L) = 1 \leq 2g + 4$. Thus, Conclusion (2) holds.

The theorem then follows from the proof of Theorem~\ref{Thm:Nonmerid bridge} except using the bound from Theorem~\ref{Bounding Bridge} in place of the inequalities from Theorems~\ref{Thm: Ess Surf} and~\ref{splitbound}. When applying Theorem ~\ref{Bounding Bridge}, note that $\boundary_0 S = \nil$, since we are in a closed manifold. Also, when $S$ is a thin or thick surface, arising from a Heegaard surface for $M'$ we can assume that $|\boundary_L S| \geq 2$, as in the proof of Theorem \ref{Thm:Nonmerid bridge}. \end{proof}

\section{Acknowledgements} The authors are grateful to the American Institute of Mathematics for its support through the SQuaREs program. The third author was also supported by NSF Grant DMS-1006369. The fifth author was supported by NSF Grant DMS-1054450. We are also grateful to a referee for numerous helpful comments and for suggesting Corollary \ref{Cor: S^3 surgeries}.


\end{document}